\documentclass[12pt]{amsart}
\usepackage{tikz}
\usetikzlibrary{arrows.meta}
\usepackage{amssymb, a4, mathrsfs}
\usepackage{cancel}
\newtheorem{them}{Theorem}[section]

\newtheorem{prop}[them]{Proposition}
\newtheorem{lem}[them]{Lemma}
\newtheorem{cor}[them]{Corollary}

\theoremstyle{definition}

\newtheorem{rem}[them]{Remark}
\newtheorem{ex}[them]{Example}
\newtheorem{defn}[them]{Definition}

\newtheorem{question}[them]{Question}
\usepackage{pstricks,pst-node,pst-plot}
\usepackage{verbatim}
\usepackage{pgf}
\usepackage{pifont}
\usepackage{bbding}
\usetikzlibrary{arrows,automata}
\usepackage[latin1]{inputenc}
\usepackage{graphicx,amscd}
\usepackage{CJK}
\usepackage{indentfirst}
\theoremstyle{definition}

\numberwithin{them}{section}
\numberwithin{equation}{section}
\newcommand{\QP}{\mbox{$\mathcal{P}$}}

\newcommand{\QQ}{\mbox{$\mathcal{Q}_\ell$}}

\newcommand{\ar}{\mathcal{R}}

\newcommand{\X}{\mathcal{X}}

\newcommand{\Y}{\mathcal{Y}}

\newcommand{\el}{\mathcal{L}}

\newcommand{\art}{\widetilde{\mathcal{R}}}

\newcommand{\elt}{\widetilde{\mathcal{L}}}

\DeclareMathOperator{\FLA}{FLA}

\begin{document}
\title{Left Ehresmann monoids with a proper basis}

\author{Gracinda M.S. Gomes}

\address{Departamento de Ci\^{e}ncias Matem\'{a}ticas \\ Faculdade de Ci\^{e}ncias\\ Universidade de Lisboa\\ 1749-016\\ Lisboa\\ Portugal}

\email{gmcunha@fc.ul.pt}

\author{Victoria Gould}

\address{Department of Mathematics\\University   of York\\Heslington\\York YO10 5GH\\UK}
\email{victoria.gould@york.ac.uk}

\author{Yanhui Wang}
\address{College of Mathematics and Systems Science\\ Shangdong University of Science and Technology\\ Qingdao\\ 266590\\ PR China}

\email{yanhuiwang@sdust.edu.cn}

\maketitle
\begin{abstract}

Left Ehresmann monoids, and their two-sided counterpart of Ehresmann monoids, were so named by Lawson, who elucidated their connection  to  the work of Ehresmann in  differential geometry. This article is dedicated to 
building a theory for  left Ehresmann monoids inspired by  that  for inverse semigroups; in  order to do so we must develop substantially different  ideas and techniques.

It is known that every left Ehresmann monoid has a cover, that is, a projection separating preimage,  of the form $\mathcal{P}_{\ell}(T,X)$, where  $\mathcal{P}_{\ell}(T,X)$ is a left Ehresmann monoid  constructed from a monoid $T$ and an order-preserving action of $T$ on a semilattice $X$ with identity.
We introduce the notion of a proper basis, and show that $\mathcal{P}_{\ell}(T,X)$, and consequently any free left Ehresmann monoid, possesses a proper basis. We show that any left Ehresmann monoid with a proper basis displays properties close to those of two-sided Ehresmann monoids.
Next, we exhibit a class of  subsemigroups $\mathcal{Q}_{\ell}(T,X,Y)$ (properly, biunary monoid subsemigroups) of the monoids $\mathcal{P}_{\ell}(T,X)$,  which are also  left Ehresmann  with a proper basis. We  prove that any left Ehresmann monoid with a proper basis is isomorphic to some $\mathcal{Q}_{\ell}(T,X,Y)$.

 Our results can be regarded as being analogous to those for proper inverse semigroups, due to McAlister and O'Carroll, the $\mathcal{Q}_{\ell}(T,X,Y)$ playing the role of the $P$-semigroups and the $\mathcal{P}_{\ell}(T,X)$ the role of the semidirect products of a semilattice by a group.
 In the process of proving our main theorems we  present a globalisation result for an order-preserving  partial action of a monoid on a partially ordered set or  semilattice. 
 
\end{abstract}

\section{Introduction}

 A major thrust of algebraic semigroup theory is the study of semigroups determined by properties of their idempotents. This article is a contribution to that area. Historically, one of  the first major suite of results in this direction was for inverse semigroups, where a semigroup $S$ is {\em inverse} if for any element $a$ in $S$ there is a unique element of $S$, denoted by $a^{-1}$, such that $a=aa^{-1}a$ and $a^{-1}=a^{-1}aa^{-1}$. Equivalently, a semigroup $S$ is inverse if for each $a\in S$ there is at least one such $a^{-1}$, that is, $S$ is {\em regular}, and the idempotents  of $S$ commute. The latter condition results in   $E(S)=\{ e\in S:e^2=e\}$ being a commutative semigroup of idempotents, that is, {\em semilattice}; the word semilattice is used since $E(S)$ may equivalently be viewed as a partially ordered set with meets.  The founding results of 
McAlister \cite{M1,M2} and O'Carroll \cite{O} together show that every inverse semigroup $S$ is in some sense close to a semidirect product $X\rtimes G$ of a semilattice $X$ acted on by a group $G$ via {\em morphisms}.
 Specifically, every inverse semigroup has a {\em cover}, that is,  an idempotent separating pre-image, which is {\em proper}. We do not give the technical definition of proper here, since an inverse semigroup is proper precisely when it has the structure of a so-called P-semigroup  precisely when it embeds into  a semidirect product of a group with a semilattice. Consequently,  every inverse semigroup is close to one that is  determined by two coordinates - one from a semilattice and one from a group.  Any free inverse semigroup is proper; here the group $G$ is a free group and the semilattice consists of finite connected subgraphs of its Cayley graph \cite{munn:72}.  These results have had lasting consequences for inverse semigroups and have motivated many extensions to  broader classes consisting of semigroups having a semilattice of idempotents, but which need not be regular. These include  (left) restriction and (left) ample semigroups \cite{F3, FGG, CG, kudryavtseva:2015, lawson:86}: without the involution $^{-1}$ the theory splits into one- and two-sided cases. Descriptions of free objects may here be formulated in terms of  Cayley graphs of free monoids; these descriptions, and other structural results, may be correspondingly phrased in terms of  semidirect products $X\rtimes M$ where here $M$ is a monoid.

This paper is concerned with left Ehresmann semigroups.   The classes of left  Ehresmann semigroups and their two-sided counterpart of Ehresmann semigroups have become a topic of recent interest from several quite different viewpoints \cite{EG, MS, stokes:21, L, KL}.  (Left) Ehresmann  semigroups form a variety of algebras possessing (one) two basic unary operations, in addition to the semigroup binary operation; that is, they are (unary) biunary semigroups. The role of the unary operations is to pick out special idempotents, the set of which form a semilattice of {\em projections}. They  first appear in the semigroup literature in a seminal paper of Lawson \cite{lawson:1991},  who elucidates the way they arise  naturally from the categories appearing in the work of Ehresmann  in differential geometry; essentially, the unary operations correspond to the domain and range in an associated category. The class of (left) Ehresmann semigroups  includes many natural examples, and strictly encompasses those classes just mentioned. In particular, the monoid of binary relations on any set is Ehresmann, but yet it does not belong to any of the classes mentioned, all of which have good representations by partial maps. But relations do not behave like maps. Consequently, the study of (left) Ehresmann semigroups requires substantial new ideas and a paradigm shift from techniques used in the former cases. The technical reason is due to the fact that certain identities - the `ample identities', which are intimately related to semidirect products and the way partial maps compose, do not hold. So here free (left) Ehresmann semigroups are described using  labelled trees \cite{K1, K2} and  semidirect products  are replaced by what may be thought of as iterated versions \cite{BGG, BGG2, BGGW}. In the left Ehresmann case, this iterated notion of a semidirect product, denoted $\mathcal{P}_{\ell}(T,X)$, is obtained from a monoid $T$ acting  on a semilattice $X$ by {\em order-preserving} maps (not necessarily morphisms).
We observe that the (left) adequate semigroups introduced by Fountain \cite{F1} are (left) Ehresmann, and form a quasi-variety of (unary) biunary semigroups, which generate the class of (left) Ehresmann semigroups. Thus the recently introduced Pretzel Monoids of \cite{HKS} are left Ehresmann.

It was known that any left Ehresmann monoid 
has a cover of the form $\mathcal{P}_{\ell}(T,X)$ \cite{GG}. What was missing in completing a theory reflecting the classic theory in the inverse case was an  analogue of the notion of a P-semigroup capturing the abstract notion of proper for inverse semigroups.  The aim of this paper is to fill this gap. 
 For technical reasons  we focus on  monoids, commenting on the semigroup case at the end of the article.   The first step is to show that any monoid of the form $\mathcal{P}_{\ell}(T,X)$ may be endowed with a second unary operation, with respect to which it is close to being Ehresmann; it belongs therefore to a class of biunary monoids we call {\em $*$-left Ehresmann}. The next  step is to show that $\mathcal{P}_{\ell}(T,X)$  has what we call a {\em proper basis}. The elements of a proper basis may be determined by two coordinates, in a way analogous to that for {\em any} element of a proper inverse semigroup. One coordinate is determined by the class of the congruence $\sigma$,
 where $\sigma$ is the least congruence identifying the projections.  What surprised us is that the second coordinate is determined by the `new' unary operation that comes into play, and not the original one given by the left  Ehresmann structure. The basis notion comes from a decomposition of elements of $\mathcal{P}_{\ell}(T,X)$ into unique  forms. This leads us to introduce monoids of the form $\QQ(T,X,Y)$ which may be regarded as 
analogues of P-semigroups. Monoids  $\QQ(T,X,Y)$ are obtained from monoids  $\mathcal{P}_{\ell}(T,X)$ by restricting the action of $T$ on $X$ to a subsemilattice $Y$. Our main result (Theorem~\ref{them:conclusion}) states, in brief:

\vspace{0.5cm}

\noindent{\bf Main Theorem} {\em Let $Q$ be a biunary monoid. Then $Q$ is a left Ehresmann  monoid with a proper basis if and only if $Q$ is isomorphic to some 
$\QQ(T,X,Y)$.}

\vspace{0.5cm}
 Theorem~\ref{them:conclusion} is stated in terms of more technical conditions, for which the motivation and terminology will be addressed subsequently.

The structure of the paper is as follows. 
In Section~\ref{sec:pre} we give the necessary background for monoids and  (left) Ehresmann monoids, including enlarging upon  the notion of $T$-normal form for a left Ehresmann monoid generated  by its projections and  a submonoid, first discussed in \cite{BGG}.
In Section~\ref{sec:*lE} we introduce the class of $*$-left Ehresmann monoids and show that any left Ehresmann monoid with uniqueness of $T$-normal forms is $*$-left Ehresmann.
The paper then takes a different direction in Section~\ref{sec:actions}, which concerns itself with the so-called globalisation of an order-preserving  partial action of a monoid on a poset or semilattice. Section~\ref{sec:actions} essentially stands alone. Although some of our results and  techniques are reminiscent of others in the literature, the extension that we require for our main results is significant enough to require a full exposition. 
Section~\ref{sec:H-proper} introduces the notions of 
 a $*$-subset, an atomic generating set, a   basis and a  proper basis, in a left Ehresmann monoid. We show that any left Ehresmann monoid with a basis is in fact $*$-left Ehresmann. We remark that the notion of a {\em proper element} is
introduced for (two-sided) Ehresmann semigroups in \cite{KL}; the viewpoint and approach of \cite{KL} are quite different from ours, but there are distinct similarities appearing at various stages. We comment further at the end of the paper. 
In Section~\ref{sec:Q} we show that monoids $\mathcal{P}_\ell(T,X)$ are left Ehresmann with a proper basis, and we define biunary monoid subsemigroups $\mathcal{Q}_\ell(T,X,Y)$ that share these properties.
Section~\ref{sec:structtheorem} pulls together many of the steps preceding it in the paper: starting with a  left Ehresmann monoid $Q$ with a proper basis, we show that the monoid $T=Q/\sigma$, where $\sigma$ is the least congruence on $Q$ generated by all pairs of projections from $Y$, has an order-preserving partial action on $Y$. This may therefore  be globalised to an action on a semilattice $X$. We then prove that  $Q$ is isomorphic to $\mathcal{Q}_\ell(T,X,Y)$, and hence sits inside $\mathcal{P}_\ell(T,X)$ as a biunary monoid subsemigroup. 
The final section draws together our results and finishes with 
some remarks and open questions.

\section{Background and preliminaries}\label{sec:pre}

The purpose of this section is to recall some standard notions related to monoids (Subsection~\ref{sub:monoids}); outline the varieties of unary and biunary monoids we will consider here, including left Ehresmann monoids (Subsection~\ref{sub:leftehresmann}); finally, discuss an important  notion of normal form for left Ehresmann monoids (Subsection~\ref{sub:normalforms}) and some consequences of possessing it.  
We refer the reader to \cite{H} for general semigroup background and to \cite{MMT:1987} for notions of universal algebra.  Many definitions we list have their counterparts for semigroups; since this paper is focused on monoids,  our default is to  present them in that context.

\subsection{Monoids and relations}\label{sub:monoids}

A monoid is a pair $(M,\cdot)$ where $M$ is a set and $\cdot$ is an associative binary operation such that $M$ possesses an identity element. We normally refer to the monoid simply as $M$, write $ab$ for $a\cdot b$, and $1$ or $1_M$ for the identity. As a (universal)  algebra, $M$ has signature $(2,0)$. We also consider unary and biunary monoids, which are monoids possessing one or two additional basic unary operations. Thus, a unary monoid has signature $(2,1,0)$ and a biunary monoid has signature $(2,1,1,0)$. Semigroups, unary and biunary semigroups, are defined similarly without requiring the existence of an identity;  we make the convention that a semigroup is non-empty. They are therefore algebras with signature $(2)$, $(2,1)$ or $(2,1,1)$, respectively. 

If $\rho$ is a (binary) relation on a set $X$, we  may write $x\, \rho\,
y$ to indicate that $(x, y) \in \rho$.    A relation $\rho$ on $X$ is an {\it equivalence relation} if it is reflexive, symmetric and transitive. If $\rho$ is an equivalence relation on $X$, then for $x\in X$ the set  $x\rho=\{ y\in X:x\,\rho\, y\}$ is the {\it $\rho$-class} or {\it equivalence class} of $x$ and we may write $x\rho$ as $[x]$. The set of all $\rho$-classes of $X$ is said to be the {\it quotient set of $X$ by $\rho$} and is denoted by $X/\rho$.

 A relation $\leq$ on $X$ is a {\it pre-order} on $X$ if it is reflexive and transitive, and  is a {\it partial order} if, in addition, it is anti-symmetric. A {\it partially ordered set} 
 or {\em poset} $(X, \leq)$ is a set $X$ together with a partial order $\leq$ on $X$; 
 we may abbreviate $(X,\leq)$ by $X$. If $x,y\in X$ possess a greatest lower bound, then this is unique and we refer to it as the {\em meet}
 $x\wedge y$ of $x$ and $y$. From a pre-order $\leq$ on a set $X$ we may construct an equivalence relation $\equiv$ on $X$ by the rule that $x\,\equiv\, y$ if and only if $x\leq y\leq x$. Then $X/\equiv$ is partially ordered by a relation, also denoted by $\leq$, where $\leq$ is defined by the rule that $[x]\leq [y]$ if and only if $x\leq y$. A subset $Y$ of a poset $X$  is an {\it order ideal} of $X$ if for all $x \in X$ and $y \in Y$  with $x \leq y$  we have that $x \in Y$.

We refer to a commutative semigroup of idempotents as a {\em semilattice}. We denote the set of idempotents of a monoid 
$M$  by $E(M)$. If $E\subseteq E(M)$ is a semilattice then we will say simply that $E$ is a {\em semilattice in $M$.} If $ef=fe$ for all $e,f\in E(M)$ then  $E(M)$ is a semilattice in $M$, and the largest such. Any semilattice may be partially ordered by the rule that $e\leq f$ if and only if $ef=e$; in this case, $ef$ is the  meet $e\wedge f$ of $e$ and $f$. On the other hand, if a poset $X$ is such that $x\wedge y$ exists for any pair of elements $x,y\in X$, then $(X,\wedge)$ is a semilattice and $x\leq y$ if and only if $x\wedge y=x$. With this in mind, we pass freely between semilattices and posets  with meets.

A {\em congruence} $\rho$ on a monoid $M$ is an equivalence relation $\rho$ such that for all $a,b,c,d\in M$, if $a\,\rho\, b$ and $c\,\rho \, d$ then $ac\,\rho\, bd$. If $\rho$ is a congruence on $M$ then $M/\rho$ becomes a {\em quotient} semigroup under the rule $[a][b]=[ab]$. 

\begin{defn}\label{defn:sigma} Let $M$ be a monoid and let $E\subseteq E(M)$. Then $\sigma_E$ denotes the least  congruence  on $M$ that contains $E\times E$. 
    \end{defn}

Following \cite[Proposition 1.5.9]{H}, we  observe that for any $a, b \in M$ we have  $a\ \sigma_E\ b$ if and only if $a = b$ or there exists  $n\in\mathbb{N}$ and a sequence\[a = c_1e_1d_1, \,c_1f_1d_1=c_2e_2d_2, \,\dots,\,  c_nf_nd_n =b\]
 where  $c_i, d_i\in M$ and  $(e_i, f_i)  \in E \times E$, for $1\leq i\leq n$.

 One can, of course, also define congruences on unary and biunary monoids, which must be compatible with the unary operation(s) in addition to the binary operation. In this article we are particularly interested in congruences identifying idempotents, and unary operations whose images are idempotents, so it will transpire that considering monoid congruences is sufficient for our purposes.

We will return to some properties of $\sigma_E$
in the next subsection.

\subsection{ Ehresmann and left Ehresmann monoids}\label{sub:leftehresmann}
We now define the classes of monoids in which our interest lies. Our approach is to consider them as  varieties of unary and biunary monoids;  there are other approaches, as we will briefly indicate. 

\begin{defn}\label{defn:lefte}
A {\em left Ehresmann monoid}
is a unary monoid $M$ satisfying the identities
\begin{equation}\label{eqn:le}
    x^{+} x=x,\, (x^+ y^+)^+ = x^+ y^+,\, x^+ y^+=y^+ x^+\mbox{ and }(xy)^+=(xy^+)^+, \end{equation} 
where  $a\mapsto a^+$  is the unary operation.\end{defn} 

It follows from Definition~\ref{defn:lefte} that the class of left Ehresmann monoids is a variety of algebras  with signature $(2,1,0)$.  Note that in a left Ehresmann monoid, we have $1^+=1$. The next lemma is essentially folklore. Since the identities for left Ehresmann monoids appear in different forms in the literature,  we  provide a brief proof for the sake of completeness.

\begin{lem}\label{lem:folk} Let $M$ be a left Ehresmann monoid. Then  $M$  satisfies the identities 
\begin{equation}\label{eqn:lem}
x^+=x^+ x^+\mbox{ and }(x^+)^+ = x^+.\end{equation}  
Consequently,  \[E=\{ a^+:a\in M\}\] is a semilattice in $M$. For any $a\in M$, the element $a^+$ is an idempotent left identity for $a$ from $E$ and is the least such under the partial order on $E$. Finally, for any $a,b,c\in M$, if $a^+=b^+$ then $(ca)^+=(cb)^+$. 
\end{lem}
\begin{proof} Let $a\in M$. From the first and second identities of (\ref{eqn:le}) we have that
\[a^+=(a^+)^+a^+=((a^+)^+a^+)^+=(a^+)^+,\]
so that the identities $x^+=x^+ x^+$ and $(x^+)^+ = x^+$ both hold and $E$ is a semilattice in $M$. Now, $a^+a=a$ and if $b^+a=a$ for some $b^+\in E$, then $a^+=(b^+a)^+=(b^+a^+)^+=b^+a^+$, using the second and fourth identities of  (\ref{eqn:le}). Thus $a^+$ is the least left identity for $a$  from $E$.  The final claim is immediate from the fourth identity of  (\ref{eqn:le}). \end{proof}

 {\em Right Ehresmann monoids} are defined dually,   so that again they form a variety of algebras with signature $(2,1,0)$. In this case 
 the unary operation is normally denoted by  $a\mapsto a^\ast$ 
and the corresponding identities are 
$xx^\ast=x$, $(x^\ast y^\ast)^\ast=x^\ast y^\ast=y^\ast x^\ast$ 
and $(xy)^\ast=(x^\ast y)^\ast$.   By duality, a  right Ehresmann monoid $M$  also  satisfies the identities $1^*=1$, $x^*=x^* x^*$ and $(x^*)^* = x^*$. It follows that  $E=\{ a^*:a\in M\}$ is a semilattice in $M$, for any $a\in M$  the element $a^*$ is the least right identity for $a$ from $E$  under the partial order on $E$, and   for any $a,b,c\in M$, if $a^*=b^*$, then $(ac)^*=(bc)^*$.

\begin{defn} \label{defn:ehresmann} An {\em Ehresmann monoid} is a biunary monoid $M$ such that $M$ is a left Ehresmann monoid under  $+$, a right Ehresmann monoid under $*$,    and satisfies the identities
\[(x^*)^+=x^*\mbox{ and }(x^+)^*=x^+.\]
\end{defn}

It follows from Definition~\ref{defn:ehresmann}  that the class of  Ehresmann monoids is a variety of algebras  with signature $(2,1,1,0)$. A consequence of
 the final set of identities is  that in an Ehresmann monoid $M$
\[\{ a^+:a\in M\}=\{ a^*:a\in M\}(=E).\] 
The next definition is therefore  unambiguous.

\begin{defn} \label{defn:proj} Let $M$ be a  (left, right) Ehresmann monoid.  The semilattice $E$ is the {\em semilattice of projections} of $M$. \end{defn}

We remark at this point that one can take a different route to defining a (left, right) Ehresmann monoid $M$, via the relations $\art_E$ and $\elt_E$, where $E$ is a semilattice in $M$. These relations are extensions of the well-known Green's relations $\ar$ and $\el$. In our context, they are respectively the kernels of the unary operations  $+$ and $*$,  respectively. For more details, we refer the reader to \cite{G}.

 The genesis of (left) Ehresmann monoids, motivated by Ehresmann's work on small ordered categories, is described in \cite{lawson:1991}. We now give some examples and special cases of (left) Ehresmann monoids.

\begin{ex}\label{ex:relations} (Binary relations and partial transformation monoids)  Let $X$ be a non-empty set and let $\mathcal{B}_X$ denote the set of binary relations on $X$. Then $\mathcal{B}_X$ is an Ehresmann monoid where the binary operation is composition of relations, the identity is the equality relation
and for $\rho\in  \mathcal{B}_X$  we have
$\rho^+=\{ (a,a):\exists (a,b)\in \rho\}$ and $\rho^*=\{ (a,a): \exists (b,a)\in \rho\}$. 

 Since Ehresmann monoids form a variety of biunary monoids, any biunary submonoid of $\mathcal{B}_X$  will also be Ehresmann. In particular, the partial transformation monoid $\mathcal{PT}_X$, equipped with the unary operations $\alpha\rightarrow \alpha^+$ and $\alpha\rightarrow \alpha^*$, where $\alpha^+$ (respectively, $\alpha^*$) is the identity in the domain (respectively, image) of $\alpha$, is  an  Ehresmann monoid.
\end{ex}

\begin{ex}\label{ex:restriction} (Restriction  monoids) A unary monoid $M$ is {\em left restriction} if it is left Ehresmann and satisfies the additional identity $xy^+=(xy)^+x$. This latter identity is significant enough to have a name: it is   known as the {\em ample identity}. 
Roughly speaking, it allows one to re-arrange products so that the idempotent elements are on the left. A monoid is left restriction if and only if it embeds (as a unary monoid) into a partial transformation monoid $\mathcal{PT}_X$,  where $\mathcal{PT}_X$ is equipped with the unary operation $+$ as defined in Example~\ref{ex:relations}.

{\em Right restriction monoids} are defined dually and a biunary monoid is said to be {\em restriction} if it is both left and right restriction.  We remark that, as in Example~\ref{ex:relations},  $\mathcal{PT}_X$ is Ehresmann but not right restriction under $*$.  \end{ex}
 
\begin{ex}\label{ex:inverse}(Inverse monoids)
 We recall that a  monoid $M$ is {\em inverse} if for every $a\in M$ there is a unique element in $M$, denoted $a^{-1}$,
such that $a=aa^{-1}a$ and $a^{-1}=a^{-1}aa^{-1}$. An inverse monoid is restriction and hence Ehresmann where
$a^+=aa^{-1}$ and $a^*=a^{-1}a$.  Note that in this case, $E=E(M)$. 
A monoid is inverse if and only if it is isomorphic to an inverse submonoid of the symmetric inverse monoid $\mathcal{I}_X$ of partial bijections on a set $X$. \end{ex}

We remark that if $M$ is a (left, right) Ehresmann monoid, then the monoid congruence  $\sigma_E$ is also a unary, or biunary monoid congruence.   Since $E$ is determined by the  basic unary operation(s) on $M$, we make the  following  definition.

\begin{defn}\label{defn:sigmaone}
Let $M$ be a (left, right) Ehresmann monoid.  We define $\sigma$ to be the (monoid) congruence  $\sigma_E$. 

\end{defn}

The least group congruence
 on an inverse semigroup is precisely $\sigma$. 
For an inverse or left restriction monoid $M$, the description of $\sigma$ simplifies to $a\, \sigma\, b$ if and only if $ea = eb$ for some $e  \in E$ (see \cite{G,H}). For left Ehresmann monoids in general, we have no such useful description. Essentially, it is the ability to move idempotents around in products,  guaranteed by the ample identity, that enables this simplification. It is the inability to do so in(left, right)  Ehresmann monoids that, here as elsewhere, gives the theory of Ehresmann monoids a very different flavour.

\begin{defn}\label{defn:reduced} A (left, right) Ehresmann monoid is {\em reduced} if $E=\{ 1\}$.
\end{defn}

We remark that if $M$ is a (left, right) Ehresmann monoid then $M/\sigma$ is reduced. In the study of (left, right)  Ehresmann monoids, reduced (left, right) Ehresmann monoids play the role held by groups for inverse monoids.

\subsection{Normal forms in left Ehresmann monoids}\label{sub:normalforms}

Let $T$ be a submonoid of a left Ehresmann monoid $M$; we mean here that $T$ is a $(2, 0)$-subalgebra of the $(2, 1, 0)$-algebra $M$. It is easy to see that $E \cup T$ generates $M$ as a left Ehresmann monoid if and only if $E \cup T$ generates $M$ as a semigroup, which we denote by $M = \langle E\cup T\rangle_{(2)}$. From now on, when we write $M = \langle E\cup T\rangle_{(2)}$,  we assume that $T$ is a submonoid of $M$.

The notions of  $T$-normal forms, and of uniqueness of $T$-normal forms, were introduced in \cite{BGG}; they are crucial to this article.  We now give the relevant material from \cite{BGG} and continue its  development.

\begin{defn}\label{defn:Tnormal} Let $M=\langle E \cup T \rangle_{(2)}$  be a left Ehresmann monoid and let $a\in M$. If
\[a= t_0e_1t_1\dots e_nt_n,\]
where $n \geq 0$, $e_1, \dots, e_n \in E\setminus \{1\}$, $t_1, \dots, t_{n-1} \in T\setminus \{1\}$, $t_0, t_n \in T$ and for $1 \leq i \leq n$,
\[e_i < (t_ie_{i+1}\dots e_nt_n)^{+}, \]
then we say that $t_0e_1t_1\dots e_nt_n$ is a {\em $T$-normal form } of $a$.

 If every element of $M$ has a unique expression in $T$-normal form, then we say that $M$ has {\em uniqueness of $T$-normal forms}. 
\end{defn}

 Strictly speaking, in Definition~\ref{defn:Tnormal}, we should say that a $T$-normal form is given by a {\em sequence} $t_0,e_1,t_1,\dots, e_n,t_n$.  For ease, we do not do this; when we say, for example, that $t_0e_1t_1\dots e_nt_n$ is a $T$-normal form, it is implicit that it is with respect to the elements mentioned.

\begin{rem}\label{rem:TTN}
The condition in  Definition~\ref{defn:Tnormal} that $e_i < (t_ie_{i+1}\dots e_nt_n)^{+}$ for $1\leq i\leq n$ may be replaced by the equivalent condition that  $e_i < (t_ie_{i+1})^+ $ for $1\leq i< n$ and $e_n<t_n^+$.
\end{rem}

\begin{rem}\label{rem:TE} Let $M=\langle E \cup T \rangle_{(2)}$  be a left Ehresmann monoid. Notice that each $t\in T$ and each $t_0e_1t_1$ with $e_1<t_1^+$ is in $T$-normal form; in particular, $1e1$ is in $T$-normal form for $e\in E\setminus\{ 1\}$. If $M$ 
has  uniqueness of $T$-normal forms, then $E\cap T=\{ 1\}$. For, if $e\in E\cap T$ and $e<1$ then $e$ and $1e1$ would both be $T$-normal forms of $e$. \end{rem}

 An algorithm is provided in \cite[Lemma 3.1]{BGG} to obtain a $T$-normal form of an element $a$ of a left Ehresmann monoid $M=\langle E \cup T \rangle_{(2)}$.  For the purposes of this article, we outline a rather more simple approach, which extracts  some additional information.

\begin{lem}\cite{BGG}\label{bgg}
 Let $M=\langle E \cup T \rangle_{(2)}$ be a left Ehresmann monoid.  
  
 (i) If $a\in M$ has an expression
 \[a=t_0e_1t_1e_2\dots e_nt_n\]
where $n \geq 0$, $e_1, \dots, e_n \in E$, $t_0, t_1,  \dots, t_{n-1}, t_n \in T$,  then $a$ has an expression in $T$-normal form as
 \[a=u_0g_1u_1g_2\dots g_ku_k\]
 where $k\leq n$. Further, if $e_n<t_n^+$, then $u_k=t_n$ and $g_k\leq e_n$.

  (ii) Any $a\in M$ can be written in $T$-normal form.  
\end{lem}

\begin{proof}  Let $a\in M$. Since $E\cup T$ generates $M$ as a semigroup, and $E$ and $T$ are both submonoids of $M$,   certainly  $a=t_0e_1t_1e_2\dots e_nt_n$
for some  $n\geq 0$ where $e_i\in E$ and $ t_j\in T$ for $1\leq i\leq n$ and $0\leq j\leq n$.  Thus if we can verify $(i)$, statement $(ii)$ follows. 

Let $a$ be as expressed  in the previous paragraph. We proceed to verify $(i)$.  We may delete any $e_i=1$, noting that if $e_n<t_n^+$, then $e_n$ is not deleted. We  may then concatenate any elements of $T$ that are now adjacent in our product, and then delete any such products that are equal to $1$, except where this is $t_0t_1\dots t_i$ or  $t_jt_{j+1}\dots t_n$. Finally, we concatenate any idempotents that are now adjacent, noticing that the product of non-identity elements of $E$ is again a non-identity element. Again we notice that if $e_n<t_n^+$, then the last term $t_n$ in the product  remains and the last idempotent is  $f=: e_ie_{i+1}\dots e_n$ so that $f\leq e_n<t_n^+$.

At this point, we have $a=s_0f_1s_1\dots f_ms_m$ with  $m\geq 0$,
 $s_1,\dots, s_{m-1}\in T\setminus \{ 1\}$, $s_0,s_m\in T$, and $f_1\dots, f_m\in E\setminus\{ 1\}$ and  $m\leq n$;  for the purposes of this proof, let us call this a {\em weak $T$-normal form}.  Further, if $e_n<t_n^+$, then $t_n=s_m$ and $f_m(=f)<t_n^+=s_m^+$.    If $m=0$ then $a\in T$ and the result is clear. Proceding by induction, suppose  that $m>0$ and the result is true for all shorter expressions of $a$ in weak $T$-normal form.   If the  expression for $a$ is in $T$-normal form, we stop. Otherwise, there is a greatest $i$ with $1\leq i\leq m$ such that $f_i\not< (s_if_{i+1}\dots s_m)^+$. 

If $(s_if_{i+1}\dots s_m)^+\leq f_i$, then
\[\begin{array}{rcl}
a&=&s_0f_1s_1\dots f_ms_m\\
&=&s_0f_1s_1\dots s_{i-1}f_i (s_if_{i+1}\dots s_m)^+s_if_{i+1}\dots s_m\\
&=&
s_0f_1s_1\dots s_{i-1} (s_if_{i+1}\dots s_m)^+s_if_{i+1}\dots s_m\\
&=&
s_0f_1s_1\dots (s_{i-1}s_i)f_{i+1}\dots s_m.\end{array}\]
If $s_{i-1}s_i\neq 1$ then we call upon induction to complete the process. Otherwise, $s_{i-1}s_i=1 $ and we delete $s_{i-1}s_i$, concatenate $f_i$ with $f_{i+1}$,  and again use induction. It is worth remarking that if $T$ has no invertible elements other than 1,  we are never in the second situation. 

Now, if $(s_if_{i+1}\dots s_m)^+\not\leq f_i$,  then
$(s_if_{i+1}\dots s_m)^+ f_i<(s_if_{i+1}\dots s_m)^+ $. We obtain $a=s_0g_1s_1\dots g_ms_m$ where $g_j=f_j$ for all $j\in \{ 1,\dots ,m\}\setminus\{i\}$ and $g_i=(s_if_{i+1}\dots s_m)^+ f_i$. We then continue to apply the above procedure, noting that the greatest $k$ with $1\leq k\leq m$ such that   $g_k\not< (s_kg_{k+1}\dots s_m)^+$ is such that $k<i$. Clearly, this process stops with  $a$ written in $T$-normal form as 
$a=u_0g_1u_1g_2\dots g_ku_k$
 where $k\leq m \leq n$.  In the case where $e_n<t_n^+$, so that also $f_m<s_m^+$, notice that this process yields a $T$-normal form with last term $u_k=s_m=t_n$ and $g_k\leq f_m\leq e_n.$  This completes the proof of the lemma. \end{proof}

 \section{$*$-left Ehresmann monoids}\label{sec:*lE}

In Section~\ref{sec:H-proper} we  recall from \cite{GG}  the structure theorem for  left Ehresmann monoids  having 
uniqueness of $T$-normal forms, which may be thought of as a generalisation of a semidirect product construction; we postpone this until we have introduced actions in Section~\ref{sec:actions}.  Our purpose here is to show that, surprisingly, a left Ehresmann monoid $M$ with uniqueness of $T$-normal forms is very close to being Ehresmann.  Specifically, in such an $M$ 
every element has a minimal right identity from $E$. We do not give the class satisfying the conditions of Lemma~\ref{lem:wra} a name here, since we will only be interested in the left Ehresmann monoids that lie  in this class.  The proof of the following may be extracted from \cite[Corollary 3.12]{G}.

\begin{lem} \label{lem:wra}
Let $M$ be a unary monoid, where the unary operation is written $a\mapsto a^*$. Then $M$ satisfies the identities
\[ xx^*=x, \, (x^*)^*=x^*,\, x^*y^*=y^*x^*\mbox{ and }(xy^*)^*y^*=(xy^*)^*  \]
if and only if  $E=\{ a^*:a\in M\}$ 
 is a semilattice in $M$ and for every $a\in M$ we have that $ a^*$ is  the least  right identity from $E$ for $a$.
\end{lem}

\begin{defn} \label{defn:starE}  A biunary onoid $M$, where the two unary operations are denoted by $+$ and $*$, is 
 $*${\em -left Ehresmann}  if it is a left Ehresmann with respect to  $+$, satisfies the identities in Lemma~\ref{lem:wra} with respect to $*$, and in addition satisfies the identities 
$(x^*)^+=x^*\mbox{ and }(x^+)^*=x^+$. 
\end{defn}

 It follows from Definition~\ref{defn:starE}  that the class of $*$-left   Ehresmann monoids is a variety of algebras of type $(2,1,1,0)$.  The class of $*$-left Ehresmann {\em semigroups} is defined using the same identities as in the monoid case, and forms a variety of algebras of type $(2,1,1)$; on one occasion, where we take a biunary monoid subsemigroup of a $*$-left Ehresmann monoid, we will need this observation.  Note that a $*$-left   Ehresmann monoid satisfies the additional identity $1^*=1$.

\begin{rem}\label{rem:starE} From Definition~\ref{defn:starE} we have that for a $*$-left Ehresmann monoid $M$,
\[E=\{ a^+:a\in M\}=\{ a^*:a\in M\}.\]In view of Lemma~\ref{lem:wra},
the extra condition that a left Ehresmann monoid must satisfy to be $*$-left Ehresmann is that  of $a^*$ being   the least right identity for $a$. However,  {\em  a $*$-left Ehresmann monoid need not be Ehresmann} since we do not have the condition that if $a^*=b^*$, then $(ac)^*=(bc)^*$ for all $c\in M$, that is, the kernel $\widetilde{\mathcal{L}}_E$ of the unary operation $*$ need not be a right congruence. 
\end{rem}

 Examples of $*$-left Ehresmann monoids abound; clearly all Ehresmann monoids, which already form a broad class, are $*$-left Ehresmann. We show in Section~\ref{sec:H-proper} that free left Ehresmann monoids are $*$-left Ehresmann, and belong to a wide class of such given by a particular construction.  However, there certainly are examples of $*$-left Ehresmann monoids that are not Ehresmann; a specific example can be seen from taking $M$ to be a  monoid that is not left cancellative, and applying the construction $\mathcal{S}(M)$ in \cite{GJ} (see Propositions 3.4 and 3.5 of \cite{GJ}).

\begin{them}\label{rightAdequate}
Let $M = \langle E \cup T\rangle_{(2)}$ be a left Ehresmann monoid with uniqueness of $T$-normal forms. Then $M$ is a $*$-left Ehresmann  monoid where  for  $a \in M$ written in $T$-normal form as $a= t_0e_1t_1\dots t_{n-1}e_{n}t_n \in M$ we have
 $$a^*=
\begin{cases}
    e_n & \mbox{if}\ t_n=1\mbox{ and }n\geq 1\\
	1 & \mbox{otherwise}.
\end{cases}$$

\end{them}

\begin{proof} Let
$a=t_0e_1t_1\dots t_{n-1}e_{n}t_n$ be in $T$-normal form. It is enough to show that $a^*$ defined as above is  the least right identity of $a$ from $E$. Clearly, $aa^*=a$.
Now assume that $af = a$ for some $f \in E$.

If $n=0$  we have $a=t_0\in T$ and $a^*=1$. If $1 \neq f \in E$, then $t_0=t_0f1$ is also a $T$-normal form for $t_0$, contradicting the uniqueness of  $T$-normal forms. Thus $1$  is the only right identity for $a$ from $E$.

Suppose now that $n\geq 1$. There are two cases to consider.

If $n\geq 1$ and $t_n=1$, then  $a^*=e_n$, and  
\[a=af =t_0e_1t_1\dots t_{n-1}e_{n}f=t_0e_1t_1\dots t_{n-1}e_{n}f1.\]
It follows that \[ a=t_0(e_1(t_1e_2\dots t_{n-1}e_nf1)^+) t_1 \dots t_{n-2}(e_{n-1}(t_{n-1}e_nf1)^+)t_{n-1}e_nf1,\]
and this must be in $T$-normal form, else, 
by  the algorithm in  Lemma~\ref{bgg},
 we would obtain a $T$-normal form for $a$  of length less than $n$, contradicting the 
uniqueness of $T$-normal forms. Therefore, again by the uniqueness,
$e_n=e_nf$ so that $e_n\leq f$, whence $e_n$ is the least right identity for $a$ from $E$.

Finally, if $n\geq 1$ and $t_n\neq 1$, then certainly $a^*=1$.  Assume that  $f \neq 1$. We have
$$a=af =t_0e_1t_1\dots e_{n}t_nf1,$$
which by uniqueness cannot be in  $T$-normal form due to its length. But  by Lemma~\ref{bgg}, since $f<1=1^+$,  we may reduce this expression to yield a $T$-normal form 
\[a=af=s_0g_1s_1\dots s_{n-1}g1.\]
Thus, once more by the uniqueness, $t_n=1$ a contradiction. Hence  $f=1$, and $1$  is the only right identity for $a$ from $E$. The result follows. 
\end{proof}

 We finish this section by quoting a result connecting left Ehresmann monoids with uniqueness of $T$-normal forms, and the congruence $\sigma$.

\begin{prop}\cite[Proposition 3.4]{BGG}\label{M/Sigma-T} 
 Let $M = \langle E\cup T\rangle_{(2)}$ be a left Ehresmann monoid with uniqueness of $T$-normal forms. Then the map $c_{T}: M \rightarrow T$ given by $c_T(a) = t_0t_1\dots t_n$, where $a = t_0e_1t_1\dots e_nt_n$, with $t_0, \dots, t_n \in T$ and $e_1, \dots, e_n \in E$, is a well-defined onto monoid morphism with ${\rm ker}\, c_T = \sigma$,  so that $M/\sigma \cong T$ as (reduced left Ehresmann) monoids.
\end{prop}

\section{Actions and partial actions of monoids on posets and semilattices}\label{sec:actions}

In this section we recall the globalisation process from \cite{meg:2004} and \cite{hollings:2007} for the strong partial action of a monoid on a set. We  show that where the partial action is by order-preserving maps on a poset (respectively, semilattice), such that the domains of the actions are order ideals, then  we may globalise to an action by order-preserving maps again on a poset (respectively, semilattice). This construction will be crucial to the proof of our main result, Theorem~\ref{them:main}.
  At the end of this section we make some remarks on the development of this theory and analogous results.

\begin{defn}\label{defn:paction} 
Let $T$ be a monoid and let $X$ be a set. Then $T$ {\em acts partially} on $X$ (on the left) if there is a partial map $T \times X \rightarrow X$, where $(t, x) \mapsto t \cdot x$, such that for all $s, t \in T$ and $x \in X$
\[\exists 1 \cdot x\mbox{ and } 1 \cdot x = x\]
and 
\[\mbox{if}\ \exists t \cdot x \mbox{ and }\exists s\cdot (t\cdot x) \mbox{ then }\exists st \cdot x\ \mbox{ and  }\ s \cdot (t \cdot x) = st \cdot x.\]

If $T$ acts partially on $X$ via the partial map $\cdot$,  then we may refer to $(X,\cdot)$ as a  {\em partial action} of $T$.  A partial action $(X,\cdot)$ is {\em full} if for all $t\in T$ there is an $x\in X$ such that $\exists t\cdot x$. 
\end{defn}

\begin{defn}\label{defn:action} If in  Definition~\ref{defn:paction} the domain of the partial 
 map $T\times X\rightarrow X$ is $T\times X$, that is, $\exists t\cdot x$   for all $t\in T$ and $x\in X$, then we say that $T$ {\em acts} on $X$ (on the left) and we may refer to $(X,\cdot)$ as an {\em action} of $T$. 
\end{defn}

Clearly,  if $(X,\cdot)$ is an action of a monoid $T$, then for all $s,t\in T$ and $x\in X$ we have $s\cdot (t\cdot x)=st\cdot x$.

\begin{defn} \label{defn:strong} Let $(X,\cdot)$ be a partial action for a monoid  $T$.   Then $(X,\cdot)$  is  {\em strong} if for all $s,t\in T$ and $x\in X$, 
\[\mbox{ if }\exists t\cdot x\mbox{ and }
\exists st\cdot x\mbox{ then }\exists s\cdot (t\cdot x)\] {(and $s \cdot (t \cdot x) = st \cdot x$).} 
\end{defn}

It is also clear that all actions are strong partial actions. The motivation to give the notion for a strong partial action comes from the following notion of globalisation \cite{meg:2004}, see also \cite{hollings:2007}.

\begin{defn} \label{defn:global}  Let $(X,\cdot)$ be a partial action of a monoid  $T$.   A {\em globalisation} of  $(X,\cdot\,)$ is  a triple $(\iota, \mathbf{X}, \ast)$, where  $\mathbf{X}$ is a set, $\iota:X\rightarrow \mathbf{X}$ is an injection 
and $(\mathbf{X},\ast)$ is an action of   $T$  such that for all $t\in T$ and $x\in X$
\[\exists t\cdot x\mbox{ if and only if }t\ast x\iota\in X\iota,\]
and, if this holds, then
\[(t\cdot x)\iota=t\ast x\iota.\] \end{defn}

\begin{them}\cite{meg:2004,hollings:2007}\label{thm:strong} A partial action of a monoid $T$ on a set has a globalisation if and only if the partial action is strong.
\end{them}

In Section~\ref{sec:structtheorem}    we require the extension  of  Theorem~\ref{thm:strong} to the case where a monoid is acting partially and strongly on a semilattice in a way that preserves the  order.  We begin with a consideration of the partial action of a monoid on a poset.

\begin{defn} \label{defn:opaction}  Let $(P, \cdot)$ be a partial action (or action) of a monoid $T$, where $P$ is a poset.   Then ($P,\cdot)$ is  {\em order-preserving }
if  for all $x,y\in P$ and for all $t\in T$, 
if $x\leq y$ and $\exists t\cdot y$ then $\exists t\cdot x$ and $t\cdot x\leq t\cdot y$.
\end{defn}

 If  ($P,\cdot)$ is an  order-preserving partial action of a monoid $T$ on a poset $P$ then, for any $t\in T$, the {\em domain of the action of $t$}, that is, the set $\{ x\in P:\exists t\cdot x\}$, is an order-ideal of $P$. 

In Definition~\ref{defn:opaction} the poset $P$ may, in fact,  be a semilattice. We first state a globalisation result for order-preserving partial actions on posets and prove this via a series of steps. We remark that some of the steps we take in our constructions  are similar to those used in \cite[Section 4]{kudryavtseva:2015}, where they occur in  a more restricted setting. Subsequently we utilise Theorem~\ref{them:orderglobal} to prove an analogous result for order-preserving partial actions on semilattices.

\begin{them}\label{them:orderglobal} Let $(P,\cdot\,)$ be a strong order-preserving partial action of a monoid $T$ on a poset $P$. 
Then there is a globalisation $(\iota, \mathbf{P}, \diamond)$ of $(P,\cdot\,)$ where $\mathbf{P}$ is a poset, $\iota$ is an order-embedding such that $P\iota$ is an order-ideal of $\mathbf{P}$,  and $(\mathbf{P},\diamond)$ is an order-preserving action of $T$.

\end{them}

We begin by recalling the standard globalisation procedure (see \cite{meg:2004} and \cite{hollings:2007}) for the partial action of $T$ on the {\em set} $P$.
The first step is to define a relation  $\equiv$ on $T \times P$ as the equivalence relation generated by the set of pairs
\[\{ \big((mn, e) , (m,n\cdot e)\big):  m,n\in T, e \in P, \exists n\cdot e \}\]
and set 
\[\Sigma := (T \times P)/\equiv.\]

To simplify notation, we write $[(t,e)]$ as $[t,e]$, for all $[(t,e)] \in \Sigma$.

\begin{lem}\label{actionFacts}\cite{meg:2004,hollings:2007}
The following statements hold for $\Sigma$ constructed as above:
\begin{enumerate}
\item[$(i)$] if  $[m,e]=[n,f]$ then $\exists m \cdot e$ if and only if  $\exists n \cdot f$, in which case $ m \cdot e = n \cdot f$;
\item[$(ii)$] there is an embedding of the  set $P$   into $\Sigma$ given by   $e\mu =[1,e] $  for all $e \in P$; 
\item[$(iii)$] there is an action $\ast$ of $T$ on $\Sigma$ given by
\[m \ast [n,e] = [mn,e];\]
\item[$(iv)$] the triple  $(\mu, \Sigma, \ast)$  is a globalisation of $(P,\cdot\, )$.  
\end{enumerate}
\end{lem}

Observe that $\mu$ being an embedding says that precisely that $[1,e] = [1,f]$ if and only if $e=f$.

The next step is to define a pre-order on $\Sigma$. Let  $\preceq$ be the relation on $\Sigma$ defined by the rule that for all $\alpha, \beta \in \Sigma$,
\[\alpha \preceq\ \beta\ \mbox{if}\ \alpha = [t, e]\mbox{ and } \beta= [t, f] \mbox{ , for some } t\in T \mbox{ and } e, f \in P \ \mbox{where}\ e \leq f. \]

Let $\leq$ be the pre-order generated by $\preceq$, that is, for all $\alpha, \beta \in \Sigma$,
\[\alpha \leq \beta \ \mbox{if and only if}\ \alpha = \alpha_0 \preceq \alpha_1 \preceq \alpha_2 \preceq \dots \preceq \alpha_n = \beta,\]
for some $\alpha_0, \dots, \alpha_n \in \Sigma$
where $n\geq 0$. 

\begin{lem}\label{lem:preservingorder} Let $[t,e],[s,f]\in \Sigma$ and $m\in T$ such that
$[t,e] \leq [s,f]$. Then \[m\ast [t,e] \leq m \ast [s,f].\]
\end{lem}
\begin{proof} It is clear from the definition that $\ast$ preserves $\preceq$, that is, if 
$\alpha\preceq \beta$ then $m\ast \alpha \preceq m\ast \beta$. The lemma now follows by bearing in mind that $\ast$ is certainly {\em well defined}.
\end{proof}

Let $\tau$ be the equivalence relation generated by $\leq$, that is, for all $\alpha, \beta \in \Sigma$,
\[\alpha\ \tau\ \beta \ \mbox{if}\ \alpha \leq \beta\ \mbox{and}\ \beta \leq \alpha.\]
The pre-order $\leq$ on $\Sigma$ now induces a partial order on  $\Sigma/\tau$, which we also denote by $\leq$. That is, for  $\alpha, \beta \in \Sigma$,
\[[\alpha] \leq [\beta] \ \mbox{if}\ \alpha \leq \beta.\]

To simplify notation, we write $\big[[t, e]\big]$ as $(t,e)$ for all $\big[[t, e]\big]\in \Sigma/\tau$.
Note that if $\exists t\cdot e$ then
\[(t,e)=\big[[t, e]\big]=\big[[1, t\cdot e]\big]=(1, t\cdot e).\]

Define  an action $\diamond$ of $T$ on $\Sigma/\tau$ by the rule that
\[m \diamond (t,e) = (mt,e).\]

\begin{lem}\label{star well defined}  The pair $(\Sigma/\tau,\diamond)$ is an order-preserving action of the monoid $T$ on the poset $\Sigma/\tau$.  
\end{lem}
\begin{proof} 
We  show that $\diamond$ is well defined. 
Suppose that $m \in T$ and $(t,e) = (s,f)$ for $(t,e),(s,f) \in \Sigma/\tau$. From $(t,e)=(s,f)$ we have  $[t, e] \leq [s, f]$ and $[s,f]\leq [t, e]$
in the pre-ordered set $\Sigma$.  Lemma~\ref{lem:preservingorder} gives that 

\[m\ast [t,e]\leq m\ast [s,f]\leq m\ast[t,e].\] Therefore $[mt, e] \leq [ms, f]\leq [mt,e]$ and consequently
\[m\diamond (t,e)=(mt, e) = (ms, f)=m\diamond (s,f),\] showing that $\diamond$ is well defined.  

It is immediate from the definition that $(\Sigma/\tau,\diamond)$ is an action of the monoid $T$. A similar argument to the above shows that if $(t,e)\leq (s,f)$ then  $m\diamond (t,e)\leq m\diamond(s,f)$, for any $m\in T$ and $(t,e),(s,f)\in \Sigma/\tau$. 
\end{proof}

 We now work towards showing the embedding $e\mapsto [1,e]$ is order-preserving. 
\begin{lem}\label{lem:firstorderideal}
For all $[s,g]\in \Sigma$ and $e \in P$, we have
\[[s,g]\leq [1,e]\Leftrightarrow \exists s\cdot g\mbox{ and }s\cdot g\leq e,\]
and then $[s,g]=[1,s\cdot g]$.

Moreover, 
\[[s,g]\leq [1,e]\Leftrightarrow [s,g]=[1,f]\mbox{ where }f\leq e.\]
\end{lem}
\begin{proof}
Suppose that 
$[s,g]\leq [1,e]$. By definition of $\leq$  we have
\[[s, g]= \alpha_0 \preceq \alpha_1 \preceq \alpha_2 \preceq \dots \preceq \alpha_n =[1,e] \]
for some $\alpha_0, \dots, \alpha_{n} \in \Sigma$. 
Now, by the definition of $\preceq$, there are $t_i\in T,
e_i,f_i\in P$ with $e_i\leq f_i$ such that
\[\alpha_0=[t_1,e_1],\, \alpha_i=[t_i, f_i]=
[t_{i+1}, e_{i+1}],\,  1\leq i\leq n-1\mbox{ and } \alpha_n=[t_n,f_n].\]  
We show by induction on $n$ that $\exists s\cdot g$ and $s\cdot g\leq e$. If $n=0$, then
$[s,g]=[1,e]$ and by Lemma~\ref{actionFacts} $(i)$ and  $(ii)$  we obtain 
$\exists s\cdot g$ and $s\cdot g=e$.

Suppose for induction that the claim is true for $n-1$. From $[t_n,f_n]=[1,e]$,  just as before 
$\exists t_n \cdot f_n$ and $t_n \cdot f_n = e$. As $e_n \leq f_n$ and $(P,\cdot)$ is an order-preserving action of $T$,  we obtain   $\exists t_n \cdot e_n$ and $t_n \cdot e_n \leq t_n \cdot f_n$. Further,
\[\alpha_{n-1}=[t_{n-1},f_{n-1}]=[t_n,e_n]=[1, t_n\cdot e_n].\] Our inductive hypothesis now
yields that $\exists s\cdot g$ and $s\cdot g\leq t_n\cdot e_n$. Hence $s\cdot g\leq t_n\cdot f_n=e$. It follows that 
 $[s,g]=[1, s\cdot g]$.

Conversely, if $\exists s\cdot g$ and $s\cdot g\leq e$, then 
\[[s,g]=[1,s\cdot g]\preceq[1,e]\]
hence $[s,g]\leq [1,e]$ as required.
This completes the proof of the first statement, and the second follows. 
\end{proof}

\begin{cor}\label{cor:equal} For any $e,f\in P$,
we  have $[1,f]\leq [1,e]$ if and only if $f\leq e$. 
\end{cor}

We now examine the induced partial order on $\Sigma/\tau$.

\begin{lem}\label{lem:harderequal} For any $(t,e), (1,f)\in \Sigma/\tau$,
we  have $(t,e)=(1,f)$ if and only if
$\exists t\cdot e$ and $t\cdot e= f$.
\end{lem}
\begin{proof} If $\exists t\cdot e$ and $t\cdot e=f$ then as noted above, 
\[(t,e)=(1, t\cdot e)=(1,f).\] 

Conversely, let $(t,e)=(1,f)$.
Then $[t,e]\leq [1,f]$ so that by Lemma~\ref{lem:firstorderideal} we have that
$\exists t\cdot e$ and $t\cdot e\leq f$. Also 
$(1,f)=(t,e)=(1, t\cdot e)$ implies 
$[1,f]\leq [1, t\cdot e]$ giving that $f\leq t\cdot e$. Thus $t\cdot e=f$.
    \end{proof}
    
 Next we show the analogue of Lemma~\ref{lem:firstorderideal}  for $\Sigma/\tau$.
\begin{lem}\label{lem:secondorderideal}
For all $(s,g)\in \Sigma/\tau$ and $e \in P$, we have
\[(s,g)\leq (1,e)\Leftrightarrow \exists s\cdot g\mbox{ and }s\cdot g\leq e,\]
and then $(s,g)=(1,s\cdot g)$.

Moreover, 
\[(s,g)\leq (1,e)\Leftrightarrow (s,g)=(1,f)\mbox{ where }f\leq e.\]
\end{lem}
\begin{proof} 
From definition, $(s,g)\leq (1,e)$ if and only if $[s,g]\leq [1,e]$, hence by Lemma~\ref{lem:firstorderideal} we get $(s,g)\leq (1,e)$ if and only if $\exists s\cdot g$ and $s\cdot g\leq e$. In this case, $(s,g)=(1, s\cdot g)$ follows.

Certainly then, if $(s,g)\leq (1,e)$, we obtain  $(s,g)=(1,f)$ where $f\leq e$. Conversely, if $(s,g)=(1,f)$ with $f\leq e$, then by Lemma~\ref{lem:harderequal} it follows that $\exists s\cdot g$ and $s\cdot g=f$, so that by the above,  $(s,g)\leq (1,e)$.
\end{proof}

\begin{cor}\label{cor:vhardequal}  For any $(1,e), (1,f)\in \Sigma/\tau$,
we  have $(1,e)\leq (1,f)$ if and only if
$ e\leq f$. In particular, $(1,e)=(1,f)$ if and only if $ e= f$.
    \end{cor}

\begin{cor}\label{cor:iota} The map
$\iota:P\rightarrow \Sigma/\tau$ given by
$e\iota=(1,e)$ is an order-embedding. 
\end{cor}
 We are now in a position to complete the proof of Theorem~\ref{them:orderglobal}.
\begin{proof} 
Let  $\mathbf{P}=\Sigma/\tau$.  From  Lemma~\ref{star well defined}, it follows that $(\mathbf{P},\diamond)$ is an order-preserving action of $T$ and by Corollary ~\ref{cor:iota}, we have that  $\iota:P\rightarrow \mathbf{P}$  is an order-embedding. Moreover, from Lemma~\ref{lem:secondorderideal}, we obtain that $P\iota$ is an order ideal of $\mathbf{P}$. Let $t\in T$ and $e\in P$.  We now check the remaining condition required for $(\iota,\mathbf{P}, \diamond)$ to be a globalisation.  Using Lemma~\ref{lem:harderequal} we get

\[\begin{array}{rcl}
t\diamond e\iota\in P\iota&\Leftrightarrow&
t\diamond (1,e)=(1,f)\mbox{ for some }f\in P\\
&\Leftrightarrow&(t,e)=(1,f)\mbox{ for some }f\in P\\
&\Leftrightarrow& \exists t\cdot e\mbox{ and }t\cdot e=f.
    \end{array}\]
    It then follows that 
    \[(t\cdot e)\iota=(1, t\cdot e)= (t,e)=  t\diamond (1,e)=t\diamond e\iota.\]
  Thus $(\iota,\mathbf{P}, \diamond)$ is a globalisation.
 \end{proof}

 Next we state and prove a partial analogue of Theorem~\ref{them:orderglobal} for the order-preserving partial action of a monoid on a {\em semilattice} and again construct our proof in a series of steps.  Note that the difference  in the statement of Theorem~\ref{them:slglobal}, compared to that of Theorem~\ref{them:orderglobal}, is that  we are not claiming that $Y\kappa$ is an order-ideal of $\mathcal{X}$. 

\begin{them}\label{them:slglobal} Let $(Y,\cdot\,)$ be a strong order-preserving partial action of a monoid $T$ on a semilattice $Y$. 
Then there is a globalisation $(\kappa, \mathcal{X}, \bullet)$ where $\mathcal{X}$ is a semilattice  with identity,  $\kappa$ is a semilattice embedding,   and the action $(\mathcal{X},\bullet)$ of $T$ is order-preserving.
\end{them}

From Theorem~\ref{them:orderglobal} there is a globalisation
$(\iota,\mathbf{Y},\diamond)$ of $(Y,\cdot\, )$, where $\mathbf{Y}$ is a poset, $\iota$ is an order-embedding such that $Y\iota$ is an order-ideal of $\mathbf{Y}$,  and the action $\diamond$ of $T$ on $\mathbf{Y}$ is order-preserving. We retain the notation of that theorem.

 Recall that the set of order-ideals of any poset $P$ forms a semilattice under intersection. For a subset $A$ of $P$,  we denote by $A^{\omega}$ 
 the order ideal of $P$ generated by $A$, that is,
 \[A^{\omega}=\{ p\in P:p\leq q\mbox{ for some }q\in A\} \]
 and for $q\in P$ we abbreviate $\{q\}^{\omega}$
 by $q^{\omega}$. It is immediate that for any $p,q\in P$ we have $p^\omega\subseteq q^\omega$ if and only if $p\leq q$ and so $p^\omega= q^\omega$ if and only if $p=q$.

Let ${\mathcal{X}}$ be the semilattice of order ideals of  ${\mathbf{Y}}$ under intersection;  the identity of ${\mathcal{X}}$ is ${\mathbf{Y}}$. 
Put
\[{\mathcal{Y}} = \{(1, e)^\omega: e\in Y\}.\]

\begin{lem}\label{EYiso}
The set $\mathcal{Y}$ is a subsemilattice  of $\mathcal{X}$ and the map $\kappa:Y \rightarrow \mathcal{Y}$ given by $e \kappa = (1,e)^\omega$ is a semilattice isomorphism. 
\end{lem}
\begin{proof} Let $e,f\in Y$ and suppose that $e\kappa=f\kappa$. Then $(1,e)^\omega=(1,f)^\omega$ so that by an earlier comment, 
$(1,e)=(1,f)$. But this says that
$e\iota=f\iota$, and so $e=f$. 

Let $e,f\in Y$. We wish to show that $(1,e)^\omega\cap (1,f)^\omega = (1, e\wedge f)^\omega$. If
$(s,g)\in (1,e)^\omega\cap (1,f)^\omega$, then
$(s,g)\leq (1,e)$ and $(s,g)\leq (1,f)$. By Lemma~\ref{lem:secondorderideal} we have 
$(s,g)=(1,h)$ where $s\cdot g=h\leq e$ and $h\leq f$. Since $Y$ is a semilattice, we get $h\leq e\wedge f$, so that
 by Corollary~\ref{cor:vhardequal} 
we obtain
$(s,g)=(1,h)\in (1, e\wedge f)^\omega$. 

Conversely, if $(t,k)\in (1,e\wedge f)^\omega$, then  by Lemma~\ref{lem:secondorderideal}, we have 
$(t,k)=(1,u)$ where  $u=t\cdot k\leq e\wedge f$. Thus $u\leq e$ and $u\leq f$, giving that
$(t,k)=(1,u)\leq (1,e)$ and $(t,k)\leq (1,f)$, that is, $(t,k)\in (1,e)^\omega\cap (1,f)^\omega$. Consequently, 
\[(1,e)^\omega\cap (1,f)^\omega=
(1,e\wedge f)^\omega.\] 
The lemma now follows.
\end{proof}

Next we define  $\bullet$ by setting, for any $m\in T$ and $I\in\mathcal{X}$,
\[m\bullet I = \{m \diamond (t,e): (t,e) \in I\}^\omega.\]

\begin{lem}\label{lem:slaction}  The pair 
  $(\mathcal{X},\bullet)$ is an order-preserving action of $T$.
 Further, for any $m\in T$ and $(t,e)\in \mathbf{Y}$
\[m\bullet (t,e)^\omega = (mt,e)^\omega.\]
\end{lem}

\begin{proof} By definition,  if $m\in T$ and $I\in\mathcal{X}$, then $m\bullet I\in\mathcal{X}$.

Let $I\in\mathcal{X}$. Since $1 \diamond (t,e)= (t,e)$ for all $(t,e)\in \mathbf{Y}$,  we have that
\[1\bullet I = \{ 1  \diamond (t,e): (t,e) \in I\}^\omega = 
 \{ (t,e): (t,e) \in I\}^\omega= I^\omega=I,\] 
because $I$ is an order-ideal.

Let $m,n\in T$. If
$(t,e)\in mn\bullet I$, then $(t,e)\leq mn\diamond (s,f)=m\diamond(n\diamond (s,f))$ for 
some $(s,f)\in I$. Certainly $(ns,f)=n\diamond (s,f)\in n\bullet I$ and $(t,e)\leq m\diamond (ns,f)$, whence
$(t,e)\in m \bullet (n\bullet I)$. Conversely, if
$(t,e)\in m \bullet (n\bullet I)$, then $(t,e)\leq m\diamond (u,p)$ for some $(u,p)\in n\bullet I$. Now 
$(u,p)\leq n\diamond (v,q)$ for some $(v,q)\in I$, and  
since the action of $T$ on $\mathbf{Y}$ is order-preserving, we get $(t,e)\leq m\diamond (n\diamond (v,q))=mn\diamond (v,q)$, whence $(t,e)\in mn\bullet I$. It follows that
$mn\bullet I=m \bullet (n\bullet I)$ and we deduce that  $(\mathcal{X}, \bullet)$ is an action of $T$ on  the underlying set of the poset $\mathcal{X}$.  
Moreover, it  is clear from the definition that  the action is order-preserving.

For the final statement, let   $m\in T$ and $(t,e)\in \mathbf{Y}$.  If
 $(s,f)\in (mt,e)^\omega$  then
 $(s,f)\leq (mt,e)=m\diamond (t,e)$ and certainly
 $(t,e)\in (t,e)^\omega$, hence $(s,f)\in m\bullet (t,e)^\omega$. Conversely, if
 $(s,f)\in m\bullet (t,e)^\omega$, we have 
$(s,f)\leq m\diamond (u,p)$ for some $(u,p)\in (t,e)^\omega$. Thus
$(u,p)\leq (t,e)$, and so
$(s,f)\leq m\diamond (t,e)=(mt,e)$. Therefore  $(s,f)\in (mt,e)^\omega$. Consequently, 
$m\bullet (t,e)^\omega = (mt,e)^\omega$. 
\end{proof}

We can now complete the proof of Theorem~\ref{them:slglobal}. 

\begin{proof} 
From Lemmas~\ref{EYiso} and \ref{lem:slaction},  The pair 
  $(\mathcal{X},\bullet)$ is an order-preserving action of $T$ and  (with some abuse of notation) $\kappa:Y\rightarrow \mathcal{X}$ is a semilattice embedding. 
Let $t\in T$ and $x\in Y$. Using Lemma~\ref{lem:slaction} 
and then Lemma~\ref{lem:harderequal},  we obtain
\[\begin{array}{rcl}
t\bullet x\kappa\in Y\kappa&\Leftrightarrow&
t\bullet (1,x)^\omega=(1,y)^\omega\mbox{ for some }y\in Y\\
&\Leftrightarrow&(t,x)^\omega=(1,y)^\omega\mbox{ for some }y\in Y\\
&\Leftrightarrow&(t,x)=(1,y)\mbox{ for some }y\in Y\\
&\Leftrightarrow& \exists t\cdot x\mbox{ and }t\cdot x=y \mbox{ for some }y\in Y.
    \end{array}\]
    It then follows that 
    \[(t\cdot x)\kappa=(1, t\cdot x)^\omega=(t,x)^\omega=t\bullet (1,x)^\omega=t\bullet x\kappa.\]
  Thus $(\kappa,\mathcal{X}, \bullet)$ is indeed a globalisation as required.  This completes the proof of the theorem.
 \end{proof}

 As promised, we give a  few comments on the development of the theory of globalisation of partial actions.  The notion of a partial action of a group was formally introduced in \cite{exel:1998} and further developed in \cite{kellendonk:2004}. The partial action of a group is always {\em strong}, in the sense that it can be globalised, and is locally invertible. In \cite{kellendonk:2004} it is shown that an order-preserving partial action of a group on a semilattice may be globalised to an action of a group on a semilattice by order-preserving maps (which must then be semilattice morphisms). A similar result, starting with a locally invertible partial action of a monoid on a semilattice, and globalising to an order-preserving action on a poset, is given by Kudryavtseva in \cite{kudryavtseva:2015}, where it is noted that the genesis of the construction may be seen as far back as \cite{munn:1976}. Our approach shares several aspects of those quoted, but is necessarily  more involved: we start with a poset and with no assumption of invertibility in our action.

Occasionally in later sections we also require the notion of a semigroup action. For convenience we state it here.

\begin{defn}\label{defn:saction} Let $S$ be a semigroup and let $X$ be a set. Then $S$ {\em acts} on $X$ (on the left) if there is a map $S\times X\rightarrow X $, such that for all $s, t \in S$ and $x \in X$ we have
$s\cdot (t\cdot x)=st\cdot x$.\end{defn}

The notion of an order-preserving semigroup action may similarly be obtained from that for a monoid.

\section{ Decompositions of  left Ehresmann monoids} \label{sec:H-proper}

In this section we introduce the class of left Ehresmann monoids with a proper basis,  which are the focus of this article. It will transpire that these monoids are in fact $*$-left Ehresmann.  The motivation comes from looking at how a proper inverse semigroup embeds into a semidirect product of a group with a semilattice. As in that case,  we must examine the action of a left  Ehresmann monoid  $M$ on the semilattice $E$ of projections, and also the partial action of the monoid
$M/\sigma$ on $E$. The difference here is that the (partial) actions are only order-preserving, rather than by morphisms. This leads us to  submonoids of  monoids of the form $\mathcal{P}_{\ell}(T,E)$,  defined in Section~\ref{sec:Q}. (The simpler notation $\mathcal{P}(T,X)$ is used in \cite{GG}. In view of subsequent work on (two-sided) Ehresmann monoids using the same convention, we prefer here  $\mathcal{P}_{\ell}(T,X)$.) The latter play the role held by semidirect products in the case of inverse semigroups, and are an  analogy thereof. The next result is well known;  a proof may be found in \cite{BGG}.  

\begin{lem}\label{lem:naturalaction}   Let $M$ be a left Ehresmann monoid with submonoid $T$. Then $T$ acts on $E$ on the left via order
preserving maps by
\[(t, e)\mapsto t \cdot e = (te)^+.\] \end{lem}
Note that in Lemma~\ref{lem:naturalaction}  the underlying set of the monoid $T$ could be that of $M$.

\subsection{Atomic subsets}

Let $M = \langle E \cup T \rangle_{(2)}$ be a left Ehresmann monoid. Since every element of $M$ is a product of elements from $E$ and $T$, and  the identity of $M$ lies in  both $E$ and $T$, we may write any element of $M$ as an alternating product of elements of $E$ and $T$, choosing to start and end with elements of $E$ or $T$. It follows that  $M=\langle H\rangle_{(2)}=\langle H^r\rangle_{(2)}$
where 
\[H=\{te: t \in T, e \in E\} \mbox{ and }H^r=\{et: t \in T, e \in E\}. \] 
It transpires, somewhat surprisingly given the motivation from semidirect products, that it is the first formulation that is the most useful. Suppose now that in addition $M$  admits uniqueness of $T$-normal forms, so that from Theorem~\ref{rightAdequate}, $M$ is also $*$-left Ehresmann. 
It will follow from a more general result in Section~\ref{sec:Q} that the subset $H$ not only generates $M$ (as a semigroup) but satisfies a number of special conditions: we call such subsets {\em atomic}. We choose the term since elements of $M$ will also admit a unique decomposition as a product of elements of $H$ satisfying conditions akin to those in the definition of $T$-normal form;  we will later refer to an atomic generating set with this latter property as a basis.

\begin{defn}\label{*-subset} Let $H\subseteq M$ where $M$ is a left Ehresmann  monoid.  We say that  $H$ is a $*$-{\it subset} of $M$ if  for every $h\in H$ there is an element in $E$, denoted by $h^*$, such that $h^*$ is the minimal right identity for $h$ from $E$.
\end{defn}

In Definition~\ref{*-subset}  we have that $hh^*=h$ and if $he=h$, where $e\in E$, then $h^*e=h^*$. Our aim is to show that with certain conditions on $H$, we can extend the $^*$-operation to $M$, in such a way that $M$ becomes $*$-left Ehresmann. Notice that $e^*=e$ for all $e\in E\cap H$. 
\begin{defn}\label{goodset}
Let $M$ be a left Ehresmann  monoid. A $*$-subset $H$ of  $M$ is said to be {\it pre-atomic}   if  the following conditions hold:
\begin{enumerate}
\item[(H1)] $E \subseteq H$;
\item[(H2)]  if $h \in H$ and $e \in  E$, then $he \in H$  and $(he)^\ast = h^\ast e$;
\item[(H3)]  if $h \in H$ and $k \in H\setminus E$, then  $h^{\ast} \geq k^{+}$ implies that $hk\in H$ and  $(hk)^\ast = k^\ast$;
\end{enumerate}

and it is said to be
 {\it atomic}  if  the next conditions also  hold:
\begin{enumerate}
\item[(H4)] for all $m \in M$ there exists $h \in H$ such that $m\sigma = h\sigma$;
\item[(H5)] if $h, k, w \in H$ with $hk\ \sigma\ w$ and $k^\ast = w^{\ast}$ then there exists $u \in H$ such that $u\ \sigma\ h$ and $u^{\ast} \geq k^+$.
\end{enumerate}
\end{defn}
 We illustrate this notion with  examples  from the theory of inverse monoids and left restriction monoids, with the usual $*$-operation.

\begin{ex}\label{ex:inverse2} Let $M$ be an inverse monoid. Then $M$ is an  atomic  subset of $M$.  To see this, recall from Example~\ref{ex:inverse} that $M$ is Ehresmann with $E=E(M)$
 and $m^*=m^{-1}m$.  Clearly (H1)  and (H4) hold. Since $M$ satisfies the identity $(xy)^*=(x^*y)^*$  we also have that (H2) and (H3) hold. Recalling that $M/\sigma$ is a group, if $h,k,w\in M$ satisfy the hypotheses of (H5), then $h\ \sigma \, hk^+=hkk^{-1}\,\sigma\, wk^{-1}$ and
$(wk^{-1})^*=kw^*k^{-1}=k^+$, so that (H5) holds with $u=wk^{-1}$.
\end{ex}

\begin{ex} \label{ex:spsd} Let $M$ be a left cancellative monoid and consider the { {\em subset expansion} $\mathcal{S}(M)$  \cite{GJ}, which is the semidirect product} obtained from the pointwise action of $M$ on the set   $\mathcal{P}_{\ell}(M)$ of subsets of $M$. We have 
\[\mathcal{S}(M)=\{ (A,a):A\subseteq M, a\in M\}\]
and with operations
\[(A,a)(B,b)=(A\cup aB,ab),\, (A,a)^+=(A,1)\mbox{ and }(A,a)^*=(a^{-1}A,1),\]
where $a^{-1}A=\{ c:ac\in A\}$. 
From \cite{GJ}, the monoid  $\mathcal{S}(M)$ is  $*$-left  Ehresmann (indeed, it is left restriction),  with $E=\mathcal{P}_{\ell}(M) \times \{1\}$ and further,
\[(A,a)\,\sigma\, (B,b)\mbox{ if and only if }a=b.\]

Again we claim that $M$ is  an  atomic  subset of $M$. We need only check (H2), (H3) and (H5). 

Let  $(A,a),(B,b)\in \mathcal{S}(M)$ so that
$((A,a)(B,b))^*=((ab)^{-1}(A\, \cup aB),1)$. Suppose that
$(A,a)^*\geq (B,b)^+$, that is, $(a^{-1}A,1)\geq (B,1)$, equivalently  $a^{-1}A\subseteq B$. Then,
 using the left cancellativity of $M$,  we have
\[\begin{array}{rcl}
c\in (ab)^{-1}(A\cup aB)&\Leftrightarrow& abc\in A\cup aB\\
&\Leftrightarrow& abc\in A\mbox{ or }abc\in aB\\
&\Leftrightarrow& bc\in a^{-1}A\mbox{ or }bc\in B\\
&\Leftrightarrow& bc\in B\\
&\Leftrightarrow& c\in b^{-1}B.\end{array}
\]
It follows that $((A,a),(B,b))^*=(B,b)^*$ and then (H3) holds. A similar calculation with $b=1$ gives (H2). To prove (H5), let  $(A,a),(B,b)\in M$  and 
consider $(aB,a) $. Then $(aB,a)\,\sigma\,  (A,a)$ and
$(aB,a)^*=(B,1)=(B,b)^+$.
\end{ex}

 \begin{rem}\label{rem:otherex} Any monoid biunary subsemigroup $S$ of $\mathcal{S}(M)$ will also satisfy (H1) to (H4) with $H=S$.  Examples abound such as 
\[\begin{array}{rcl}
\mathcal{S}_f(M)&=&\{ (A,a)\in \mathcal{S}(M):|A|<\infty\}\\
$Sz$(M)&=&\{ (A,a)\in \mathcal{S}(M): 1,a\in A,\, |A|<\infty\}\\
\FLA(X)&=&\{ (A,a)\in \mathcal{S}(X^*): a\in A,\, |A|<\infty \mbox{ and }A=A^{\downarrow}\}.\end{array}\]
Here for a set  $X$ and for $U\subseteq X^*$ we denote by $U^{\downarrow}$  the set of prefixes of elements of $U$. The monoid Sz$(M)$ is the {\em Szendrei expansion} of $M$ (see \cite{Sz}) and $\FLA(X)$ is the free left ample monoid on $X$ (see \cite{F3}). We note that the identity of 
 $\FLA(X)$ is $(\{ 1\},1)$. 

  It is easy to see that in these three cases  we have subsemigroups closed under $+$, and the first two share the identity of $\mathcal{S}(M)$. A straightforward calculation gives that  $a^{-1}A$ is finite if $A$ is, contains $1$ if $a\in A$ and is prefix closed if $M=X^*$ and $A$  is. It then follows that in each case,  the given semigroup is closed under $*$ too. 
  
  It is also clear that $\mathcal{S}_f(M)$ satisfies (H5), using $(aB,a)$ again.
When $M$ is free, 
then Sz$(M)$ satisfies (H5), and so does $\FLA(X)$ where  we replace  
$(aB,a)$  by $(a^{\downarrow}\cup aB,a)$ in both cases. 
\end{rem}

We now utilise the structure of $\FLA(X)$ to a greater degree, to exhibit the  behaviour we will see for  FLE$(X)$.

\begin{ex}\label{ex:frela}  We  consider $\FLA(X)$ for some set $X$. For $a\in X^*$ we put
$\overline{a}=(a^{\downarrow},a)$ and we abbreviate
$(B,1)\in E$ by $\overline{B}$. Let 
\[H=\{ \overline{a}\, \overline{B}: a\in X^*, (B,1)\in E\}\subseteq \FLA(X),\]
noting that $\overline{a}\, \overline{B}=(a^{\downarrow}\cup aB,a)$ and
$X^*\cong \{ \overline{a}:a\in X^*\}$.
We claim that $H$ is  an    atomic  subset of $FLA(X)$. 

  It is easy to check that $\overline{a}\, \overline{B}=\overline{c}\, \overline{D}$ 
if and only if $a=c$ and $B=D$, so that elements of $H$   have a unique decomposition in the given form.  To compute $((a^{\downarrow},a)(B,1))^*$ notice that  
$a^{-1}(a^{\downarrow}\cup aB)=a^{-1}(aB)=B$ so that $(\overline{a}\, \overline{B})^*=\overline{B}$. 

Clearly (H1) holds since
$\overline{B}=\overline{1}\, \overline{B}$. To see that (H3) holds, let $\overline{a}\, \overline{B}\in H$ and  $\overline{c}\, \overline{D} \in H\setminus E$,
so that $c\neq 1$. If $\overline{B} =(\overline{a}\, \overline{B})^*\geq (\overline{c}\, \overline{D})^+$
then $(\overline{a}\, \overline{B})(\overline{c}\, \overline{D})=\overline{a}\, \overline{c}\overline{D}=
\overline{ac}\, \overline{D}\in H$ and
$(\overline{ac}\, \overline{D})^*=\overline{D}=(\overline{c}\, \overline{D})^*$.
The argument for (H2) is similar and (H4) follows from the description of $\sigma$ in Example~\ref{ex:spsd} (inherited by $\FLA(X)$). To see that (H5) holds, it is sufficient to notice that for any 
$\overline{a}\, \overline{B}\in H$, we have
$\overline{a}\, \overline{B} \,\sigma\, \overline{a}$ and $\overline{a}^*=1_{\FLA(X)}\geq (\overline{c}\, \overline{D})^*=\overline{D}$ 
for any $\overline{c}\, \overline{D}\in H$. 
\end{ex}

\subsection{$H$-proper}\label{sub:proper}

 We introduce the concept of $H$-proper for a left Ehresmann monoid $M$ with generating $*$-subset $H$, by analogy with the concept of proper for inverse and (left) restriction monoids. The surprise here is that we use $*$ and not $+$.

\begin{defn}\label{defn:hproper} 
Let $M$ be a left Ehresmann monoid and  let $H$ be $*$-subset of $M$. Then $H$ is {\em proper} if for all  $h, k \in H$,  \[h^*=k^*\mbox{ and }h\,\sigma\, k\mbox{ if and only if }h=k.\]
We say that  $M$ 
is {\em $H$-proper} if $M=\langle H\rangle_{(2)}$ for some proper subset $H$.\end{defn}

The reader should compare the above with \cite[Definition 3.4]{KL} in the two-sided case, where the aim is to find `matching factorisations'. In this one-sided case, we must take a different approach.

\begin{rem} \label{ex:proper} We return to Examples~\ref{ex:inverse2}, \ref{ex:spsd} and \ref{ex:frela}. An inverse monoid with $H=M$ is proper if and only if it is proper (or, $E$-unitary), as an inverse monoid. For $\mathcal{S}(M)$ with $M$ left cancellative and $H=\mathcal{S}(M)$ we can certainly find pairs $(A,a),(B,b)$ such that  
$(A,a)\,\sigma\, (B,b)$, that is, $a=b$, and
$(A,a)^*=(B,b)^*$ without $A=B$ (take $A,B$ with $A\neq B$ and $a^{-1}A=a^{-1}B=\emptyset$), whence $\mathcal{S}(M)$ is not $H$-proper. However, these examples are all proper as {\em left restriction monoids}, that is if $(A,a)\,\sigma\, (B,b)$ and $(A,a)^+=(B,b)^+$, then $(A,a)=(B,b)$.  On the other hand,
in Example~\ref{ex:frela} with $H$ as given, if  
$\overline{a}\, \overline{B}\,\sigma\,\overline{c}\, \overline{D}$, that is, $a=c$, and $(\overline{a}\, \overline{B})^*=(\overline{c}\, \overline{D})^*$, that is, $B=D$, then $\overline{a}\, \overline{B}=\overline{c}\, \overline{D}$, so that $\FLA(X)$ is $H$-proper.
\end{rem}

\begin{lem}\label{lem:proper}   Let $M$ be  a left Ehresmann monoid with a $*$-subset $H$ that satisfies condition (H2).
Then $H$ is proper if and only if for all $h,k\in H$, 
\[h\,\sigma\, k\mbox{ if and only if }hk^*=kh^*.\] 
\end{lem}

\begin{proof}  Suppose that $H$ is proper. Let $h,k\in H$. If $hk^*=kh^*$ then clearly $h\,\sigma\, k$.
Assume now that $h\,\sigma\, k$. Then $hk^*\,\sigma\, kh^*$. From (H2)  we have $hk^*,kh^*\in H$ and
$(hk^*)^*=h^*k^*=k^*h^*=(kh^*)^*$. As $H$  is proper,  we obtain $hk^*=kh^*$. 

Conversely, assume  that for all $h,k\in H$ we have
$h\,\sigma\, k$ if and only if $hk^*=kh^*$. Suppose that $h,k\in H$, $h^*=k^*$ and $h\,\sigma\, k$. 
Then $h=hh^*=hk^*=kh^*=kk^*=k$, and so $H$ is proper. 
    \end{proof}

\begin{cor}\label{cor:hSigma1_M}   Let $M$ be  a left Ehresmann monoid with a $*$-subset $H$ that  is proper and   satisfies condition (H2).
Then for all $h \in H$,  \[h\, \sigma\, 1\mbox{ if and only if }h \in E.\]
\end{cor}
\begin{proof}
Suppose that $h \in H$. Clearly if $h \in E$ we have  $h\, \sigma\, 1$. Conversely, if $h\, \sigma\, 1$, then from Lemma~\ref{lem:proper} as $H$ is proper we have $h1^* = 1h^*$, that is, $h =h^*$. So $h \in E$.
\end{proof}

\subsection{$H$-canonical forms  and bases}\label{sub:h-simple}
The following observation lies at the root of our next concept.
Let $M = \langle E \cup T \rangle_{(2)}$  be a left Ehresmann monoid which admits uniqueness of $T$-normal forms, so that $M$ is $*$-left Ehresmann by Theorem ~\ref{rightAdequate}.
 
 Any $m \in M$ can be written in $T$-normal form, that is $m=t_0e_1t_1\dots e_nt_n\in M$ with $n \geq 0$, $e_1, \dots, e_n \in E\setminus \{1\}$, $t_1, \dots, t_{n-1} \in T\setminus \{1\}$, $t_0, t_n \in T$ and $e_i<(t_ie_{i+1}\dots e_nt_n)^{+}$ for $1 \leq i \leq n$. Thus putting
\[H=\{te: t \in T, e \in E\}\]
we get $M=\langle H\rangle_{(2)}$.

Also, recall from Remark~\ref{rem:TTN} that the latter condition is equivalent to saying that $e_n<t_n^+$ and   $e_i < (t_ie_{i+1})^+$, for all $1 \leq i < n$. With $m$ as given, note that  $t_{i-1}e_i1$ is in   $T$-normal form for $1\leq i\leq n$. By Theorem~\ref{rightAdequate}, we have
$(t_{i-1}e_i)^* = e_i$ for $1\leq i\leq n$.
Hence 
\[(t_0e_1)^* < (t_1e_2)^+,\ (t_1e_2)^* < (t_2e_3)^+,\ \dots,\ (t_{n-1}e_n)^* < t_n^+ = (t_n1)^+.\]

 Consider $t_ie_{i+1}$ where $1 \leq i \leq n-1$. If $t_ie_{i+1}\in E$ then $t_i\,\sigma \, t_ie_{i+1}\, \sigma\,  1$, so that $t_i=1$ from Proposition~\ref {M/Sigma-T}, contradicting the fact that $t_i\neq 1$. Hence $t_ie_{i+1}\notin E$.

The above considerations lead us to present the  next notion.

\begin{defn}\label{defn:hnormal}  Let $M=\langle H\rangle_{(2)}$  be a left Ehresmann monoid  where  $H$ is a $*$-subset.
 An element $m = h_1\dots h_n$  of $M$ is in {\em $H$-canonical form} (with length $n$) if $h_1, \dots, h_n \in H$,  
\[h_1^{\ast} < h^{+}_2, \dots, h^{\ast}_{n-1} < h^{+}_n\]
and  $h_i\notin E$ for  $2 \leq i \leq n$. 
 If $n$ is least for the element $m$  then we say $ h_1\dots h_n$ is a {\em short} $H$-canonical form for $m$.
\end{defn}

In Definition~\ref{defn:hnormal}  if $H$ is 
 pre-atomic  and proper then  Corollary~\ref{cor:hSigma1_M} tells us the condition that  $h_i\in E$ is equivalent to the condition that $h_i\,\sigma\, 1$. 
Similarly to our convention for $T$-normal form, strictly speaking  we should distinguish the element $h_1\dots h_n$ of $M$ from the sequence $h_1,h_2,\dots,h_n$; however, we do not take this step and no ambiguity should arise.

{\begin{defn}\label{defn:hcanonical} Let $M=\langle H\rangle_{(2)}$  be a left Ehresmann monoid  where  $H$ is a  $*$-subset.  
 If an element  of $M$ has a unique expression in $H$-canonical form, this is called  its {\em $H$-canonical form}.  When every  element of $M$ has such a unique expression, we say that $M$ has {\em $H$-canonical forms}. 
\end{defn}

{\begin{defn}\label{defn:basis}  Let $M=\langle H\rangle_{(2)}$  be a left Ehresmann monoid  where  $H$ is a  $*$-subset. 
  If $H$ is (pre-)atomic and $M$ has {\em $H$-canonical forms} then we say that $H$ is a {\em (pre-)basis} for $M$.
\end{defn}

\begin{rem}\label{ex:notgood}  We briefly return to Examples~\ref{ex:inverse2}, \ref{ex:spsd} and \ref{ex:frela} in the light of $H$-canonical forms. 
If $H=M$ as in Examples~\ref{ex:inverse2} and \ref{ex:spsd}
then we do not have   $M$-canonical forms (although we trivially have uniqueness of short $M$-canonical forms). 
In Example~\ref{ex:frela}, for $\FLA(X)$, we observe that we do not even have that   short $H$-canonical forms are unique. To see this, consider $X=\{ x,y\}$ and
$(A,x^2)\in \FLA(X)$ where $A=\{ 1,x,x^2,xy\}$.
Then $(A,x^2)=(\overline{1}\, \overline{A})(\overline{x^2}\, \overline{1})$ and $(\overline{1}\, \overline{A})^*=\overline{A}< \overline{\{1,x,x^2\}}
=(\overline{x^2}\overline{1})^+$.  But also  $(A,x^2)=(\overline{x}\overline{\{ 1,x,y\}})(\overline{x}\overline{\{ 1\}})$ and 
$(\overline{x}\overline{\{ 1,x,y\}})^*=\overline{\{ 1,x,y\}}< \overline{\{ 1,x\}}=(\overline{x}\overline{\{ 1\}})^+$.   We will see that in the corresponding case of  FLE$(X)$ we do have uniqueness of $H$-canonical forms for a directly corresponding $H$; essentially this is because the ample identity is not working behind the scenes. We comment further in Remark~\ref{rem:ample}. 
\end{rem}

We now explain a process to produce an $H$-canonical form from a product of elements in an pre-atomic subset $H$.

\begin{lem}\label{h+m-Hormal form}
 Let $M=\langle H\rangle_{(2)}$  be a left Ehresmann monoid  where  $H$ is a  pre-atomic $*$-subset. 
 Suppose that $h \in H$ and $m = h_1h_2\dots h_n\in M$ is in $H$-canonical form. Then $hm$ can be written as an expression in $H$-canonical form as follows:
$$ hm=
	\begin{cases}
        (hh_1)h_2\dots h_n & \mbox{if\ either}\ h_1\in E\mbox{ or}\ (h_1\notin E\ \text{and}\ h^*h_1^+= h_1^+)\\
        (hh_1^+)h_1\dots h_n & \mbox{if}\  h_1\notin E\mbox{ and } h^*h_1^+< h_1^+.\\
	\end{cases}$$
\end{lem}

\begin{proof} We have three cases, as follows.

(i) If $h_1\in E$ then by  (H2), we get $hh_1 \in H$ and $(hh_1)^*=h^*h_1\leq h_1=h_1^* < h_2^+$. So $(hh_1)h_2\dots h_n$ is in $H$-canonical form.

 (ii) If  $h_1\notin E $ and $h^*h_1^+ = h_1^+$, i.e. $ h^*\geq h_1^+ $,  then by (H3) we get  $hh_1 \in H$  and $(hh_1)^*=h_1^* < h_2^+$. So  in this case we again have $(hh_1)h_2\dots h_n$ is in $H$-canonical form.
 
 (iii) If $h_1\notin E$ and $h^*h_1^+ < h_1^+$,  then again using (H2) we have $(hh_1^+)h_1\dots h_n$ in $H$-canonical form.
\end{proof}

Let $M=\langle H\rangle_{(2)}$ be a left Ehresmann monoid with pre-atomic $*$-subset $H$.  Suppose that  $m = h_1h_2\dots h_n \in M$ where $h_1, h_2, \dots, h_n \in H$. \, Lemma~\ref{h+m-Hormal form} effectively gives  an algorithm for reducing $m$ to an expression in $H$-canonical form.

Clearly $h_n$ is in $H$-canonical form. Working from right to left, assuming we have written $h_i\dots h_n$ in $H$-canonical form, say as
$k^i_1\dots k^i_{n(i)}$, then we use the above process to write $h_{i-1}k^i_1\dots k^i_{n(i)}$ in $H$-canonical form as $k^{i-1}_1\dots k^{i-1}_{n(i-1)}$.

The next corollary follows directly from Lemma~\ref{h+m-Hormal form} 
 using induction on $n$.

\begin{cor}\label{cor:short}  
Let $M=\langle H\rangle_{(2)}$ be a left Ehresmann monoid where $H$ is a  pre-atomic $*$-subset. Suppose that  $m = h_1h_2\dots h_n$ where $h_i\in H$ for $1\leq i\leq n$. Then $m$ has an $H$-canonical form $m=k_1\dots k_p$ where $p\leq n$. 
\end{cor}

\begin{cor}\label{cor:short2}    
Let $M=\langle H\rangle_{(2)}$ be a left Ehresmann monoid with  pre-basis $H$. 
Suppose that  
$m\in M$  has $H$-canonical form  of length n. 
 If  $m=h_1\dots h_n$ where $h_i\in H$ for $1\leq i\leq n$,
 then 
\[m=(h_1(h_2\dots h_n)^+)(h_2(h_3\dots h_n)^+)\dots (h_{n-1}h_n^+)h_n\]
 is its $H$-canonical form. 
\end{cor}

\begin{proof} 

We  make frequent use of Corollary~\ref{cor:short} and the fact that $H$-canonical forms are unique, 
which  together tell us that we cannot write $m$ as a  product of {\em fewer} than $n$ elements of $H$.  
If $m=h_1\dots h_n$ it is clear that 
\[m=(h_1(h_2\dots h_n)^+)(h_2(h_3\dots h_n)^+)\dots (h_{n-1}h_n^+)h_n\]
The result is obvious if $n=1$. Suppose therefore that $n>1$. Notice first that $h_n\notin E$, 
otherwise $h_{n-1}h_n\in H$  by (H2) and 
as $m=h_1\dots h_{n-2}(h_{n-1}h_n)$,  the length of its $H$-canonical form would be less or equal to $n-1$ by Corollary~\ref{cor:short}, a contradiction. 
Further,  
 $h_i(h_{i+1}\dots h_n)^+\in H$ but $h_i(h_{i+1}\dots h_n)^+\notin E$ for  $2\leq i\leq n-1$, else we  could consider the element  $h_{i-1}h_i(h_{i+1}\dots h_n)^+$,  use  (H2) once more,
 and shorten the length of $H$-canonical form of $m$ again.

For $1\leq i\leq n-1$, using the second condition of (H2), we get 
\[(h_i(h_{i+1}\dots h_n)^+)^*=h_i^*(h_{i+1}\dots h_n)^+\leq (h_{i+1}\dots h_n)^+=(h_{i+1}(h_{i+2}\dots h_n)^+)^+,\] for
$1\leq i\leq n-1$, 
where the factor $(h_{i+2}\dots h_n)^+$ is  empty  for $i=n-1$. If $(h_i(h_{i+1}\dots h_n)^+)^*=(h_{i+1}(h_{i+2}\dots h_n)^+)^+$, then  by (H3) 
\[(h_i(h_{i+1}\dots h_n)^+)(h_{i+1}(h_{i+2}\dots h_n)^+)\in H,\] and a  reduction in the number of elements of $H$ in the expression of $m$ can be made, giving a contradiction again. 
So $(h_i(h_{i+1}\dots h_n)^+)^*<(h_{i+1}(h_{i+2}\dots h_n)^+)^+$.  The result follows.\end{proof}

\begin{lem}\label{some help}
Let $M=\langle H\rangle_{(2)}$ be a left Ehresmann monoid with pre-basis $H$. Let $m=h_1\dots h_n$ be in $H$-canonical form and let $e\in E$.  Suppose that $m=me$.  Then the last term of the $H$-canonical form of $me$ is $h_ne$, so that $h_ne=h_n$. 
\end{lem}
\begin{proof}
 By (H2), the element $h_ne$ is in $H$  and so the decomposition 
$h_1\dots h_{n-1}(h_ne)$ of  $m=me$  has  $n$ elements of $H$. Using Corollary~\ref{cor:short2}
and the fact that  $H$ is a pre-basis, we may conclude that the last term of the $H$-canonical form of $me$ is $h_ne$. 
 
\end{proof}

\begin{lem}\label{prop}\label{prop:extra} 
Let $M=\langle H\rangle_{(2)}$ be a left Ehresmann monoid with 
pre-basis $H$. Then the following conditions hold:

\begin{enumerate}
    \item[$(i)$] if  $h_1\dots h_n$ is  in $H$-canonical form, then $(h_1\dots h_n)^+=h_1^+$.
   
    \item[$(ii)$]  if $h \in H$ and $k \in H\setminus E$, then  $h^{\ast} \geq k^{+}$ if and only if  $hk\in H$, and then $(hk)^\ast = k^\ast$.
\end{enumerate}
\end{lem}

\begin{proof} $(i)$ Suppose that $m=h_1\dots h_n$ is in $H$-canonical form. To see that $m^+=h_1^+$ we use induction, noticing that if $h_1\dots h_n$ is in $H$-canonical form, then so is $h_i\dots h_n$ for any $1\leq i\leq n$. If $n=1$ it is clear that $m^+=h_1^+$. Suppose for induction that
$(h_2\dots h_n)^+=h_2^+$. Then 
\[m^+=(h_1\dots h_n)^+=(h_1(h_2\dots h_n)^+)^+
=(h_1h_2^+)^+=(h_1h_1^*h_2^+)^+=(h_1h_1^*)^+=h_1^+,\]
using the facts that $M$ satisfies the identity $(xy)^+=(xy^+)^+$ and as $m$ is in $H$-canonical form we have that $h_1^*<h_2^+$.

$(ii)$ Take $h\in H$ and $k\in H\setminus E$. The direct implication is given by condition (H3). Assume that
$hk\in H$. By (H2) we have  $hk^+\in H$ with $(hk^+)^*=h^*k^+\leq k^+$. If $h^*k^+ < k^+$ then 
 $(hk)=(hk^+)k$  where both expressions are  $H$-canonical forms, contradicting uniqueness of such. Therefore $h^*k^+ =k^+$, i.e. $k^+\leq h^*$.
    \end{proof}

Next we show that  for a left Ehresmann monoid $M=\langle H\rangle_{(2)}$, with pre-basis $H$   we can extend the $*$-operation defined on $H$ to the whole of $M$, obtaining a $*$-left Ehresmann monoid.

\begin{them}\label{prop:total} 
Let $M=\langle H\rangle_{(2)}$ be a left Ehresmann monoid with 
pre-basis $H$.  Then $M$ is a $*$-left  Ehresmann monoid, where  for  $m=h_1\dots h_n$  in $H$-canonical form, 
    $$(h_1\dots h_n)^\ast = h_n^\ast.$$
  \end{them}  
\begin{proof} With $*$ as in the statement,   we
need to show that the identities of Definition ~\ref {defn:starE} hold, namely

\[ xx^*=x, \, (x^*)^*=x^*,\, x^*y^*=y^*x^*\mbox{ and }(xy^*)^*y^*=(xy^*)^*  \] and 
$$(x^*)^+=x^*\mbox{ and }(x^+)^*=x^+.$$

Except for the identity $(xy^*)^*y^*=(xy^*)^*$, all the others are clear consequences of  the definition 
of a $*$-subset. To prove the identity $(xy^*)^*y^*=(xy^*)^*$,  let $u,v\in M$ and let $e=v^*$. Write
$ue=h_1\dots h_n$  in $H$-canonical form. As $(ue)e=ue$ Lemma ~\ref{some help} gives that $h_ne=h_n$ and from the definition of $*$-subset, we therefore have $h_n^*e=h_n^*$. Then 
$$(uv^*)^* = 
(ue)^*=h_n^*=h_n^*e=(uv^*)^*v^*,$$
 as required.
\end{proof}

\begin{rem}\label{rem:ample}(The left ample identity) We will see in the remainder of the paper 
canonical examples of
$*$-left Ehresmann monoids with $H$-canonical forms for a suitable atomic generating set $H$. These  include FLE$(X)$, where the $H$ taken corresponds directly to that for $\FLA(X)$. The reason that $\FLA(X)$ does {\em not} have $H$-canonical forms stems from the fact that it satisfies the {\em ample identity} 
$xy^+=(xy)^+x$. If a left  Ehresmann monoid $M=\langle E\cup T\rangle_{(2)}$ satisfies this identity, then every element of $M$ can be written in the form   $es$, where $e\in E$ and $s\in T$. Further, if also $f\in E$ and $t\in T$, then $(es)(ft)=(e(sf)^+)st$.  Inverse, left ample and left restriction monoids all satisfy the ample  identity and this leads to descriptions involving semidirect products. In general, left Ehresmann monoids do not satisfy the ample identity. Nevertheless, for $*$-left Ehresmann monoids, we may use $H$-canonical forms to determine their structure using semigroups $\mathcal{P}_{\ell}(T,X)$ in the place of semidirect products. 
\end{rem}

\section{Monoids $\mathcal{P}_\ell(T,X)$ and $\QQ(T, X, Y)$} \label{sec:Q}

The first aim of this section is  to recall the construction of the left Ehresmann monoids $\mathcal{P}_{\ell}(T,X)$. Since any such monoid has uniqueness of $T$-normal forms, it follows that $\mathcal{P}_{\ell}(T,X)$ is in fact $*$-left Ehresmann. We show that $\mathcal{P}_{\ell}(T,X)$ has  a  proper  basis.  We then find certain monoid subsemigroups of 
$\mathcal{P}_{\ell}(T,X)$, denoted $\QQ(T, X, Y)$,  that also have a  proper  basis.  Notice that monoids $\QQ(T, X, Y)$  are determined by data strongly reminiscent of that used to construct a McAlister $P$-semigroup \cite{M1,M2}. 

\subsection{The monoids $\mathcal{P}_{\ell}(T,X)$ }\label{sub:PTX}
 First, we outline the construction of the  {\em semigroup free product}  $U\ast  V$ of semigroups $U$ and $V$. We view   $U\ast V$  as a union of two  disjoint monoids (parts) and then set 
$$U\ast V:= \{ a_1 \dots a_m: m \geq 1 \mbox{ and } \forall i,  a_i \in U\cup V  \mbox{ and }  a_i, a_{i+1} \mbox{not in the same part} \},$$
where the product of elements $u=a_1 \dots a_m$ and $v=b_1 \dots b_k$ is given  by \, $a_1 \dots a_mb_1 \dots b_k$ if $a_m$ and $b_1$ are in different parts, or by $a_1 \dots a_{m-1}cb_2 \dots b_k$, where $c=a_mb_1$ if $a_m$ and $b_k$ are in the same part. For further details, see \cite[Chapter 8]{H}.

 Let $T$ be a monoid,   
 $X$  a semilattice with identity
and  $(X,\cdot)$  an order-preserving action of $T$. Our convention will be that $1_T$ (respectively $1_X$) denotes the identity of $T$ (respectively of $X$).   
Let $T * X$ be the  {\em semigroup} free product of $T$ and $X$. Since $T$ acts on the left of $X$ via order-preserving maps, there is a monoid morphism
\[\phi: T \rightarrow {\mathcal{O}}^{\ast}_X, \ (t\phi)(y) = t \cdot y,\]
where ${\mathcal{O}}^{\ast}_X$ is the monoid of order-preserving maps of $X$ and a $ ^{\ast} $ denotes the dual of a monoid, so that in ${\mathcal{O}}^{\ast}_X$, maps are composed from right to left. The monoid semilattice $X$ acts on the left of itself by order-preserving maps via multiplication, so that there is a monoid morphism, also denoted $\phi$, given by
\[\phi: X \rightarrow {\mathcal{O}}^{\ast}_X,\ (z\phi)(y) = zy.\]
By the universal property of free products, we have a semigroup morphism
\[\phi: T * X \rightarrow {\mathcal{O}}^{\ast}_X\]
defined by
\[(s_1\dots s_n)\phi = s_1\phi \dots s_n\phi,\]
where each $s_i \in T \cup X$. We thus have  an order-preserving  action of the semigroup $T * X$ on $X$, which we may also without ambiguity denote by $\cdot$, so that
\[s_1\dots s_n \cdot y = s_1 \cdot (s_2 \cdot (\dots (s_n \cdot y))).\]
 Now for $\omega \in T * X$, we define
\[\omega^{+} = \omega \cdot 1_X,\]
so that $e^{+} = e$, for all $e \in X$.

\begin{them}\cite[Theorem 2.2, Corollary 3.4]{GG}\label{P1} 
Let $T$ be a monoid and $(X,\cdot)$  an order-preserving action of $T$ on the left of a semilattice $X$ with identity. Let the unary operation  ${+}$ be defined on $T * X$ as above. Consider the semigroup congruence $\sim$ on $T * X$ generated by
\[K = \{(\alpha^{+}\alpha, \alpha): \alpha \in T * X\} \cup \{(1_T, 1_X)\}.\]
Let $ \QP_{\ell}(T, X) = (T * X)/\sim$. Then $\QP_{\ell}(T, X)$ is a left Ehresmann monoid with $[\alpha]^{+} = [\alpha^{+}]$, identity $[1_X]$ and semilattice of projections
\[\overline{X} = \{[y]: y \in X\}.\]
Let $\overline{T} = \{[t]: t \in T\}$. Then $X$  is isomorphic to $\overline{X}$ and $T$ is isomorphic to $\overline{T}$ under restrictions of the natural morphisms. Further, 
\[\mathcal{P}_{\ell}(T,X)=\langle \overline{X}\cup\overline{T}\rangle_{(2)}\] and  $\QP_{\ell}(T, X)$ has uniqueness of $\overline{T}$-normal forms. Finally, we have
$\mathcal{P}_{\ell}(T,X)/\sigma=\overline{T}/\sigma$, that is, every $\sigma$-class of $\QP_{\ell}(T, X)$ contains a unique element of $\overline{T}$.

Conversely, if $M=\langle E\cup T\rangle_{(2)}$ is a left Ehresmann monoid  such that $M$ has uniqueness of $T$-normal forms, then $M\cong\mathcal{P}_{\ell}(T,E)$.
\end{them}

Note that from Theorem~\ref{rightAdequate}, any left Ehresmann monoid of the form $\mathcal{P}_{\ell}(T,X)$ is in fact $*$-left Ehresmann. Observe also that from \cite[Theorem 4.5]{GG}, the free left  Ehresmann monoid FLE$(X)$ on any set $X$  is of the form $\mathcal{P}_{\ell}(X^*,E)$. 

\begin{cor}\label{cor:free} The free left  Ehresmann monoid FLE$(X)$ on any set $X$  is $*$-left  Ehresmann.
\end{cor}

\subsection{Proper bases for monoids $\mathcal{P}_{\ell}(T,X)$}\label{sub:PTXagain}  

Let  $\mathcal{P}_{\ell}(T,X)$ be a  monoid as defined as in Subsection~\ref{sub:PTX}. We recall from Theorem~\ref{P1} and the remark following that $\mathcal{P}_{\ell}(T,X)=\langle \overline{X}\cup\overline{T}\rangle_{(2)}$ is a $*$-left Ehresmann monoid with 
$\overline{T}$-normal forms,
 where $\overline{X}\cong X$ and $\overline{T}\cong T$. 
Conversely, if $M=\langle E\cup T\rangle_{(2)}$ is a left Ehresmann monoid with  $T$-normal forms, then $M\cong\mathcal{P}_{\ell}(T,E)$  and  $M$ is $*$-left Ehresmann. For ease of notation, in this section we denote an element of $\mathcal{P}_{\ell}(T,X)$ as $\alpha$  (where $\alpha \in T\ast X)$  rather than $[\alpha]$, so that  $[x]\in \overline{X}$ and $[t]\in \overline{T}$ are denoted by $x$ and $t$, respectively. Consequently, we  identify $\overline{X}$ with $X$ and $\overline{T}$ with $T$.

\begin{defn}\label{defn:H} For a monoid $\mathcal{P}_{\ell}(T,X)$ we set
\[H^{\mathcal{P}}=: \{ te:t\in T, e\in X\}.\]\end{defn}

 Following similar abuses of notation, whenever we write expressions such as $te\in H^{\mathcal{P}}$, it will be understood that $t\in T$ and $e\in X$. 
Since $\mathcal{P}_{\ell}(T,X)=\langle X\cup T\rangle_{(2)}$  
and, with our notational convention, both $X$ and $T$ contain the identity of  $\mathcal{P}_{\ell}(T,X)$, we get
$\mathcal{P}_{\ell}(T,X)=\langle H^{\mathcal{P}}\rangle_{(2)}$. 
We  show that $H^{\mathcal{P}}$ is a proper   basis for $\mathcal{P}_{\ell}(T,X)$. We begin with  the key properties of  elements of $H^{\mathcal{P}}$. 

\begin{lem}\label{lem:HH}  Let  $h=te\in H^{\mathcal{P}}$. Then
\begin{enumerate} 
\item[$(i)$]
$h^+=t\cdot e\mbox{ and }h^*=e$;
\item[$(ii)$] $h\,\sigma\, 1_X$ if and only if $t=1_T$;
\item[$(iii)$] 
$h\in X$ if and only if $t=1_T$;
\item[$(iv)$]   $te=sf\in H^{\mathcal{P}}$ if and only if $t=s$ and $e=f$.\end{enumerate} 
\end{lem}
\begin{proof}  The fact that $h^+=t\cdot e$ follows from Theorem~\ref{P1} and the definition of $^+$ used there. If $e=1_X$, then $h=t$  is in $T$-normal form so that by Theorem~\ref{rightAdequate} we have $h^*=1_X=e$. If $e\neq 1_X$ then $h=te1_T$ is in $T$-normal form and again Theorem~\ref{rightAdequate} gives that $h^*=e$. Thus $(i)$ holds.

As  $h\,\sigma\, t$ we get $h\,\sigma\, 1_X$ if and only if $t\,\sigma\, 1_X$, but this is equivalent to $t=1_T$.  Thus $(ii)$ holds and $(iii)$ is a consequence.

Considering $(iv)$, if $te=sf$, then $t\,\sigma\, s$ and so $t=s$, and  $e=(te)^*=(sf)^*=f$. The converse is clear. 
\end{proof}

\begin{cor}\label{cor:proper} The subset $H^{\mathcal{P}}$ of $\mathcal{P}_{\ell}(T,X)$ is proper.
\end{cor}
\begin{proof} Let $te,sf\in H^{\mathcal{P}}$.
If $te\,\sigma\, sf$ and $(te)^*=(sf)^*$, then $t\,\sigma\, s$ so that $t=s$ by Theorem~\ref{P1}, also $e=f$ by Lemma~\ref{lem:HH} so that $te=sf$.   \end{proof}

We now present a technical lemma that we can subsequently use to leverage  the fact that
$\mathcal{P}_{\ell}(T,X)$ has $T$-normal forms to  deduce that it has $H^{\mathcal{P}}$-canoncial forms.

\begin{lem}\label{T-normal vs T-Y notmal}
Let  $\alpha = h_1\dots h_n\in \mathcal{P}_{\ell}(T,X)$ where $h_i=t_{i-1}e_i$, for $ 1\leq i\leq n$. Then $\alpha$ is in $H^{\mathcal{P}}$-canonical form if and only if either $e_n\neq 1_X$ and $t_0e_1t_1e_2\dots e_n1_T$   is in $T$-normal form, or $e_n = 1_X$  and $t_0e_1t_1e_2\dots e_{n-1}t_{n-1}$ is in $T$-normal form, where if $n \geq 2$ then $t_{n-1} \neq 1_T$.
\end{lem}

\begin{proof}
Suppose that $\alpha= t_0e_1t_1 \dots t_{n-1}e_n$ is in 
$H^{\mathcal{P}}$-canonical form.  Then
\[h_1^*<h_2^+,\, h_2^*<h_3^+\,\dots,\, h_{n-1}^*<h_n^+\]
and $h_i\notin X$ for $2\leq i\leq n$. 
Since $h_i\notin X$ for $1<i\leq n$, we have $t_{i-1}\neq 1_T$ for $1< i\leq n$, so that $t_1,\dots, t_{n-1}\in T\setminus \{ 1_T\}$. On the other hand, if $e_i=1_X$ for $1\leq i<n$, then $h_i^*=1_X$, contradicting $h_i^*<h_{i+1}^+$.
Thus $e_1,\dots, e_{n-1}\in X\setminus\{ 1_X\}$. For $1\leq i<n$ we get $e_i=h_i^*<h_{i+1}^+=(t_ie_{i+1})^+$.

If $e_n\neq 1_X$, then $e_n<1_T^+=1_X$, hence using 
Remark~\ref{rem:TTN}
we obtain \, $t_0e_1t_1\dots e_n1_T$  in $T$-normal form. Also,  if $e_n=1_X$,  we get $t_0e_1\dots e_{n-1}t_{n-1}$ in $T$-normal form for similar reasons, and if $n\geq 2$ then $t_{n-1}\neq 1_T$.

Conversely, suppose that $e_n \neq 1_X$ and $t_0e_1t_1e_2\dots t_{n-1}e_n1_T$ is in $T$-normal form. Then
\[e_1, \dots, e_{n} \in X\setminus \{1_X\},\,   t_1, \dots, t_{n-1} \in T\setminus \{1_T\},\]
and for $1\leq i\leq n$, we have $e_i<(t_ie_{i+1})^+$
and $e_n<1_T^+=1_X$. 
 By Lemma~\ref{lem:HH}, we obtain $h_i\notin X$ for $2\leq i\leq n$. Further, $h_i^*=e_i<(t_ie_{i+1})^+=h_{i+1}^+$ for $1\leq i<n$. Thus
 $h_1\dots h_n$ is in $H^{\mathcal{P}}$-canonical form.

 Finally, suppose that $e_n=1_X$ and $t_0e_1t_1\dots e_{n-1}t_{n-1}$ is in $T$-normal form, where if $n\geq 2$ then $t_{n-1}\neq 1_T$. Hence $e_1,\,\dots, e_{n-1}\in X\setminus\{ 1_X\}$ and $t_1,\dots, t_{n-2}\in T\setminus \{ 1_T\}$. Lemma~\ref{lem:HH} gives  $h_i\notin X$ for $2\leq i\leq n$. Further, for
 $1\leq i<n$ we obtain  $h_i^*=e_i<(t_ie_{i+1})^+
 =h_{i+1}^+$ so that again, $h_1\dots h_n$ is in $H^{\mathcal{P}}$-canonical form. 
\end{proof}

\begin{cor}\label{cor:uniquenesshsim}  The monoid 
$\mathcal{P}_{\ell}(T,X)$ has $H^{\mathcal{P}}$-canonical forms.
\end{cor}
\begin{proof} 
The existence of $H^{\mathcal{P}}$-canonical forms is guaranteed by Theorem~\ref{P1} and the discussion at the beginning of 
Subsection~\ref{sub:h-simple}. 
Suppose that 
\[
\alpha=h_1\dots h_n=k_1\dots k_m\]
are two expressions for $\alpha$ in $H^{\mathcal{P}}$-canonical form.  Let $h_i=t_{i-1}e_i$ and  $k_j=s_{j-1}f_j$ for
$1\leq i\leq n$ and $1\leq j\leq m$. We now call upon Lemma~\ref{T-normal vs T-Y notmal}. 

If $e_n\neq 1_X$ and $f_m\neq 1_X$, then
\[\alpha=t_0e_1t_1e_2\dots e_n1_T=s_0f_1s_1f_2\dots f_m1_T\]
are two expressions for $\alpha$ in $T$-normal form. By uniqueness of $T$ normal forms we have $n=m$, $t_i=s_i$ for $0\leq i\leq n-1$ and $e_i=f_i$ for $1\leq i\leq n$, so that $h_i=k_i$ for $1\leq i\leq n$. A similar argument holds if $e_n=f_m=1_X$. 

If $e_n\neq 1_X$ and $f_m=1_X$ then 
\[\alpha=t_0e_1t_1e_2\dots e_n1_T=s_0f_1s_1\dots f_{m-1}s_{m-1} \]
are two expressions for $\alpha$ in $T$-normal form. If $m=1$, then $\alpha=s_0$ which is in $T$-normal form, contradicting the first expression for $\alpha$ as such. Thus $m\geq 2$ so that $s_{m-1}\neq 1_T$. But this contradicts uniqueness of $T$-normal forms. Hence this case, and the dual, cannot happen.
\end{proof}

\begin{lem}\label{prop:atomicH^P} 
The subset $H^{\mathcal{P}}$ is atomic.\end{lem}
\begin{proof} Since $X\cup T\subseteq H^{\mathcal{P}}$ it is clear that (H1) holds.  If $h=te$ and $f\in X$, then
$hf=tef\in H^{\mathcal{P}}$ and $(hf)^*=ef=h^*f$, thus (H2) holds.  Suppose now that $h=te,k=sf\in H^{\mathcal{P}}$. If $e=h^*\geq k^+$, then $hk=tesf=tsf\in H^{\mathcal{P}}$ and, by Lemma~\ref{lem:HH}, we get 
$(hk)^*=f=k^*$ whence (H3) holds. Condition (H4) follows from the fact that regarding $T$ as a submonoid of  $\mathcal{P}_{\ell}(T,X)$ we have  $T/\sigma=\mathcal{P}_{\ell}(T,X)/\sigma$. Finally, (H5) is in this case immediate, since for any $h=te\in H^{\mathcal{P}}$ we have $h\,\sigma\, t$ and $t^*=1_X$. Thus $H^{\mathcal{P}}$ is atomic. 
\end{proof}

 We draw together the results of this subsection.

\begin{them}\label{them:newform} 
Let $T$ be a monoid  and let $(X,\cdot)$ be an order-preserving action of $T$ on the left of a semilattice $X$. The $*$-left Ehresmann monoid $\mathcal{P}_{\ell}(T,X)$ has proper basis  $H^{\mathcal{P}}=\{ te:t\in T, e\in X\}$. 

For  any $\alpha\in \mathcal{P}_{\ell}(T,X)$ where $\alpha= h_1\dots h_n$ in $H^{\mathcal{P}}$-canonical form, 
\[\alpha^+=h_1^+\mbox{ and }\alpha^*=h_n^*.\]
\end{them}
\begin{proof}The first statement follows from Corollary~\ref {cor:proper}, Corollary~\ref{cor:uniquenesshsim} and Proposition~\ref{prop:atomicH^P}. The final statement 
is $(i)$ in Lemma~\ref{prop:extra}.
\end{proof}

\subsection{Monoids $\QQ(T, X, Y)$ } \label{subsec:Q}

 We now fix a  monoid $\mathcal{P}_{\ell}(T,X)$, and 
examine  some particular monoid subsemigroups, denoted by  $\QQ(T, X, Y)$. These monoids  transpire to also possess a proper  basis, which will be a subset of $H^{\mathcal{P}}$.

\begin{defn} Consider a monoid $\mathcal{P}_{\ell}(T,X)$.   Let  $Y$ be a subsemilattice of $X$ with identity $1_Y$.  Suppose that
the following conditions hold for all $t\in T$ and $e,f\in Y$:

(A) if  $e \leq f$, then $t \cdot f \in Y$ implies $t \cdot e \in Y$;

(B) there exists $g \in Y$ such that $t \cdot g \in Y$.

Let 
\[H^{\mathcal{Q}}=\{ te:t\in T, e\in Y\mbox{ and }t\cdot e\in Y\}\]
and \[\QQ(T, X, Y)=\langle H^{\mathcal{Q}}\rangle_{(2)}. \]
\end{defn}

To simplify the notation, we may write simply $\QQ$ instead of $\QQ(T, X, Y)$.

\begin{rem}\label{rem:restriction} In the above context, 
if  $Y$ is {\em any} subsemilattice of $X$, 
then $(Y,\cdot)$ is a partial action of $T$ on the {\em set} $Y$, where in this case $\exists t\cdot y$ for $y\in Y$ if and only if $t\cdot y\in Y$. Since this is the {\em restriction} to $Y$ of the action $(X,\cdot)$ it follows from Theorem ~\ref{thm:strong} that this 
 partial 
action is strong. Condition (A) above tells us that $(Y,\cdot)$ is an order-preserving partial action of $T$, since the domains of the 
partial
action are order-ideals of $Y$.  Condition (B) says that the partial action is full.

Conversely, let $Y$ be a semilattice with identity and let $(Y, \cdot)$ be an order-preserving full partial action of a monoid $T$. By Theorem~\ref{them:slglobal} there is a globalisation $(\kappa,\X,\bullet)$ of $(Y, \cdot)$ where $\X$ is a semilattice with identity and $(\X,\bullet)$ is an order-preserving action of $T$.  We may then construct the  monoid $\mathcal{P}_\ell(T,\X,\Y)$. Since $(\kappa,\X,\bullet)$ is a globalisation of $(Y,\cdot)$ it follows that Conditions (A) and (B) hold for the restriction of the action $(\X,\bullet)$ to $\Y=Y\kappa$. We may then construct the subsemigroup $\QQ(T,\X,\Y)$ of $\mathcal{P}_\ell(T,\X,\Y)$ where $\Y$ is $Y\kappa$. 
\end{rem}

Notice that $H^{\mathcal{Q}}\subseteq H^{\mathcal{P}}$. Further, $Y\subseteq \QQ$ and if $h=te\in H^{\mathcal{Q}}$, then 
$h^*=e\in Y$ and $h^+=t\cdot e\in Y$. Moreover, it follows from Lemma~\ref{lem:HH} that  for $h\in H^{\mathcal{Q}}$ we have  $h\in X$ if and only if $t=1_T$ if and only if $h\in Y$.

\begin{lem}\label{lem:lucky} Let $\alpha\in \QQ$. Then 
the $H^{\mathcal{P}}$-canonical form of $\alpha$ is a product of elements of $H^{\mathcal{Q}}$. Further,
$H^{\mathcal{P}}\cap \QQ=H^{\mathcal{Q}}$ and $X\cap \QQ=Y$. 
\end{lem}
\begin{proof}

We start by showing that $X\cap H^{\mathcal{Q}}=Y$. Clearly $Y\subseteq X\cap H^{\mathcal{Q}}$. 
If $e\in  X\cap H^{\mathcal{Q}}$, then $e=1_Te=tf$ for some $f\in Y$. By 
Lemma~\ref{lem:HH} $(iv)$ we have that $e=f$, so that 
  $e\in Y$.

Let $\alpha=h_1\dots h_n$ where $h_i\in H^{\mathcal{Q}}$. We show by induction that $\alpha$ has $H^{\mathcal{P}}$-canonical form $\alpha=k_1\dots k_m$
where $k_i\in H^{\mathcal{Q}}$ for $1\leq i\leq m$.

If $n=1$ the result is clear, since $H^{\mathcal{Q}}\subseteq H^{\mathcal{P}}$. 

Suppose now that the result is true for all elements of $\QQ$ that are products of $n-1$ elements of $H^{\mathcal{Q}}$. Then
$h_2\dots h_n=k_2\dots k_m$ say, where $k_2\dots k_m$
is in $H^{\mathcal{P}}$-canonical form and $k_i\in H^{\mathcal{Q}}$ for $2\leq i\leq m$. 

We have  $\alpha=h_1(k_2\dots k_m)$.
 Observe that  $k_2\in X$ if and only if $k_2\in Y$, since $X\cap H^{\mathcal{Q}}=Y$ and $k_2\in H^{\mathcal{Q}}$.
 We now discuss the various cases given by  Lemma~\ref{h+m-Hormal form}. Let $h_1=te$ for $t\in T$ and $e\in Y$, so that also $t\cdot e\in Y$ as $h_1\in H^{\mathcal{Q}}$.

If $k_2\notin Y$ and $h_1^*<k_2^+$ then $h_1k_2\dots k_m$  is in $H^{\mathcal{P}}$-canonical form.

If $k_2\in Y$ or $k_2\notin Y$ and $e=h_1^*\geq k_2^+$, then
$\alpha=(h_1k_2)k_3\dots k_m$ is in $H^{\mathcal{P}}$-canonical form.  If $k_2=f\in Y$,
then as $ef\in Y$, $ef\leq e$ and $t\cdot e\in Y$, we obtain $t\cdot ef\in Y$
by (A), so that $h_1k_2=tef\in H^{\mathcal{Q}}$. On the other hand, if $k_2\notin Y$, say $k_2=sf$ with $s\cdot f\in Y$, as   $e=h_1^*\geq k_2^+=s\cdot f$, we have $h_1k_2=tek_2=tk_2=tsf$. Now
$ts\cdot f=t\cdot (s\cdot f)$ with  $s\cdot f\in Y$, $t\cdot e\in Y$ and  $s\cdot f\leq e$, so that again $t\cdot (s\cdot f)\in Y$, hence $h_1k_2\in H^{\mathcal{Q}}$. 

Finally, if $k_2\notin Y$ and $h_1^*,k_2^+$ are incomparable, then $\alpha=(h_1k_2^+)k_2\dots k_m$  is in $H^{\mathcal{P}}$-canonical form, and  $h_1k_2^+=tek_2^+$.  Again
$ek^+_2\in Y$ and $ek^+_2\leq e$ and as $t\cdot e\in Y$ we have $t\cdot ek_2^+\in Y$, so that $tek_2^+\in H^{\mathcal{Q}}$. This completes the proof of the first statement.

It is clear that  
$H^{\mathcal{Q}}\subseteq H^{\mathcal{P}}\cap \QQ$. For the converse, consider $h\in H^{\mathcal{P}}\cap \QQ$. Then $h$ has an $H^{\mathcal{P}}$-canonical form which is a product of elements of $H^{\mathcal{Q}}$. Since $h$ is already in $H^{\mathcal{P}}$-canonical form it follows that $h\in H^{\mathcal{Q}}$. Thus $H^{\mathcal{Q}}\supseteq H^{\mathcal{P}}\cap \QQ$ and hence we have equality. 

Clearly $Y\subseteq  X\cap \QQ$.
If $e\in X\cap \QQ$, then $e\in H^{\mathcal{P}}\cap \QQ$ so that $e\in X\cap H^{\mathcal{Q}}=Y$, therefore  $Y= X\cap \QQ$.
\end{proof}

\begin{lem}\label{lem:biunarysubsemigroup} The semigroup $\QQ$ is a biunary subsemigroup of $\mathcal{P}_{\ell}(T,X)$ and is a monoid with identity $1_Y$. Consequently, $\QQ$ is a 
$*$-left Ehresmann monoid with semilattice of   projections $Y$.
\end{lem}
\begin{proof} We have $1_Y\in H^{\mathcal{Q}}$. Let 
$h=te\in H^{\mathcal{Q}}$. Clearly $h1_Y=h$.
On the other hand, $h^+=t\cdot e\in Y$ so that
$1_Y h=h$. Thus $1_Y$ is a two-sided identity for the elements of the generating set $H^{\mathcal{Q}}$  of  $\QQ$, and hence it is the identity for $\QQ$. 

Let $\alpha=h_1\dots h_n\in \QQ$ be in $H^{\mathcal{P}}$-canonical form, where $h_i\in H^{\mathcal{Q}}$ for $1\leq i\leq n$. By Theorem~\ref{them:newform} we have 
$\alpha^*=h_m^*$ and
$\alpha^+=h_1^+$. It follows from Lemma~\ref{lem:HH} that
$\alpha^*, \alpha^+\in Y$. 

The final statements hold since  $*$-left Ehresmann semigroups form a variety of biunary semigroups and $Y\subseteq H^{\mathcal{Q}}$.
\end{proof}

 To argue    that $H^{\mathcal{Q}}$ is a proper basis for $\QQ$, we must first determine $\sigma$ in $\QQ$.  For the purposes of the next lemma we denote the congruence 
    $\sigma$ in $\mathcal{P}$ (respectively $\QQ$) by
    $\sigma^{\mathcal{P}}$ (respectively $\sigma^{\mathcal{Q}}$).

\begin{lem}\label{aSigmac(a)e}
If $\alpha = h_1h_2\dots h_n\in \QQ$ is in $H^{\mathcal{P}}$-canonical form, where
$h_i=s_ie_i\in H^{\mathcal{Q}} $ for $1\leq i \leq n$, then 
$\alpha\, \sigma^{\mathcal{Q}}s_1\dots s_ne$ for some $e \in Y$, where also $s_1s_2\dots s_n \cdot e \in Y$.
\end{lem}
\begin{proof} We proceed by induction. If $n=1$ then the result is clear. Suppose for induction that the result is true for $n-1$ where $n >1$. Then \[s_2e_2\dots s_ne_n\, \sigma^{\mathcal{Q}} \, s_2s_3\dots s_n e\] for some $e \in Y$ where $s_2s_3\dots s_n \cdot e \in Y$. Notice that
$s_2\dots s_ne\in H^{\mathcal{Q}}\subseteq \QQ$. By Condition (B) there exists $f \in Y$ such that $s_1s_2\dots s_n \cdot f \in Y$, whence
$s_1s_2\dots s_n  f\in H^{\mathcal{Q}}\subseteq \QQ$.
Now as $ef \leq e$ and $ef \leq f$, by Condition (A) we get $s_2s_3\dots s_n \cdot ef \in Y$ and $s_1s_2\dots s_n\cdot ef \in Y$, thus $s_2s_3\dots s_n ef,  s_1s_2\dots s_n ef \in H^{\mathcal{Q}}$. Then
\[
\begin{aligned}
 \alpha=(s_1e_1)(s_2e_2\dots s_ne_n)\,& \sigma^{\mathcal{Q}}\, (s_1e_1)(s_2s_3\dots s_ne)\\
&\sigma^{\mathcal{Q}}\, (s_1e_1)(s_2s_3\dots s_ne)f\\
&=(s_1e_1)(s_2s_3\dots s_nef)\\
&=(s_1e_1)(s_2s_3\dots s_nef)^+(s_2s_3\dots s_nef)\\
&=s_1(s_2s_3\dots s_nef)^+e_1(s_2s_3\dots s_nef).\end{aligned}
\]
Next notice that
\[s_1\cdot (s_2\dots s_nef)^+=
s_1\cdot (s_2\dots s_n\cdot ef)=
s_1\dots s_n\cdot ef\in Y\]
and so $s_1(s_2\dots s_nef)^+\in H^{\mathcal{Q}}$. Therefore
\[\alpha\,\sigma^{\mathcal{Q}}\,s_1(s_2s_3\dots s_nef)^+(s_2s_3\dots s_nef) =s_1\dots s_nef.\]
The result follows by induction.
\end{proof}

\begin{cor}\label{cor:PSigma=QSigma}
If $\alpha, \beta\in \QQ$ then $\alpha\, \sigma^{\mathcal{Q}}\, \beta$ if and only if $\alpha\, \sigma^{\mathcal{P}}\, \beta$.
\end{cor}
\begin{proof}
 Suppose that $\alpha =h_1\dots h_n$ and $\beta=k_1\dots k_m$ where
 $h_i=s_ie_i$ and $k_j=t_jf_j$ for 
 $1\leq i\leq n$ and $1\leq j\leq m$ are elements of $H^{\mathcal{Q}}$.  Assume that  $\alpha\, \sigma^{\mathcal{P}}\, \beta$. It follows from Theorem~\ref{P1} and Proposition~\ref{M/Sigma-T}  that $s_1s_2\dots s_n = t_1t_2\dots t_m$, which we write as $w$. By Lemma~\ref{aSigmac(a)e} we have $\alpha\, \sigma^{\mathcal{Q}}\, we$ and $\beta\, \sigma^{\mathcal{Q}}\, wf$ for some $e, f \in Y$ where $w \cdot e \in Y$ and $w \cdot f \in Y$. Now \[we\,\sigma^{\mathcal{Q}}\, wef=wfe\,\sigma^{\mathcal{Q}}\, wf\]
 and so $\alpha \, \sigma^{\mathcal{Q}}\, \beta$. The converse is clear.
\end{proof}

\begin{them}\label{them:Q} 
Let $T$ be a monoid  and  $(X,\cdot)$  an order-preserving action of $T$ on the left of a semilattice $X$.  Let  $Y$ be a subsemilattice of $X$, with identity.  Suppose that
the following conditions hold for all $t\in T$ and $e,f\in Y$:

(A) if  $e \leq f$, then $t \cdot f \in Y$ implies $t \cdot e \in Y$;

(B) there exists $g \in Y$ such that $t \cdot g \in Y$.

Let 
\[H^{\mathcal{Q}}=\{ te:t\in T, e\in Y\mbox{ and }t\cdot e\in Y\}\]
and  \[\QQ(T, X, Y)=\langle H^{\mathcal{Q}}\rangle_{(2)}. \]

Then $\QQ(T, X, Y)$ is a $*$-left Ehresmann 
 biunary  monoid subsemigroup of \, 
 $\mathcal{P}_{\ell}(T,X)$, and $H^{\mathcal{Q}}$ is a proper  basis of  $\QQ(T, X, Y)$. 
\end{them}
\begin{proof} We have shown in Lemma~\ref{lem:biunarysubsemigroup}  that   $\QQ$  is $*$-left Ehresmann biunary monoid subsemigroup of $\mathcal{P}_{\ell}(T,X)$, with semilattice of  projections $Y$. It remains to show that $H^{\mathcal{Q}}$ has the properties claimed.

Certainly $Y\subseteq H^{\mathcal{Q}}$ so that (H1) holds. In view of Lemma~\ref{lem:lucky} conditions (H2) and (H3) hold since they are true for $H^{\mathcal{P}}$. For (H4), if
$\alpha=h_1\dots h_n$ where $h_i=s_ie_i$, $1\leq i\leq n$, then by Lemma~\ref{aSigmac(a)e}, we have 
$\alpha\,\sigma^{\mathcal{Q}}\,\ s_1\dots s_ne$, for some $e\in Y$ with $s_1\dots s_n\cdot e\in Y$, that is, 
$s_1\dots s_ne\in H^{\mathcal{Q}}$.  

For (H5), let $h=te, k=sf$ and $w=ug$ be elements of $H^{\mathcal{Q}}$ with $hk\,\sigma^{\mathcal{Q}}\, w$ and $k^*=w^*$.  By Lemma~\ref{lem:HH} we get $f=g$.  Corollary~\ref{cor:PSigma=QSigma} gives  $tesf=hk\,\sigma^{\mathcal{P}}\, w=ug$, and from Theorem~\ref{P1} it follows that  $ts=u$. Then 
$t\cdot (sf)^+=t\cdot (s\cdot f)= ts\cdot f=u\cdot f =u\cdot g\in Y$, whence  
$t(sf)^+\in H^{\mathcal{Q}}$ 
and $(t(sf)^+)^*=(sf)^+=k^+$. Further,
$te\, \sigma^{\mathcal{P}}\, t(sf)^+$ and again by Corollary~\ref{cor:PSigma=QSigma}, we obtain 
$te\, \sigma^{\mathcal{Q}}\, t(sf)^+$.  Thus  (H5) holds, and therefore $H^{\mathcal{Q}}$ is  atomic. 

The fact that $\QQ$ is $H^{\mathcal{Q}}$-proper follows from Theorem~\ref{them:newform}, which gives that $\mathcal{P}_{\ell}(T,X)$ is  $H^{\mathcal{P}}$-proper, together with Corollary~\ref{cor:PSigma=QSigma}. 

Let $\alpha\in \QQ$. It is clear 
from Lemma~\ref{lem:lucky}  that the unique $H^{\mathcal{P}}$-canonical form of  $\alpha$ is an $H^{\mathcal{Q}}$-canonical form. On the other hand, if $\alpha=h_1\dots h_n$ is in $H^{\mathcal{Q}}$-canonical form then this will be in $H^{\mathcal{P}}$-canonical form unless some $h_i\in X$, for $2\leq i\leq n$. But if $h_i\in X$, then $h_i\in X\cap \QQ=Y$, by Lemma~\ref{lem:lucky}, contradicting $h_1\dots h_n$ being an $H^{\mathcal{Q}}$-canonical form. We conclude  that $\QQ$ has   $H^{\mathcal{Q}}$-canonical forms. \end{proof}

\section{A structure theorem }\label{sec:structtheorem}

The aim of this section is to establish the converse to Theorem~\ref{them:Q}.  To this end, let $Q$
be a left Ehresmann monoid  such that $Q=\langle H\rangle_{(2)}$ for a proper basis $H$.  We show that $Q$ is isomorphic to  a  biunary monoid subsemigroup  $\QQ(T,\X,\Y)$ of some $\mathcal{P}_{\ell}(T,\X)$ where $T = Q/\sigma$ and $\Y$ is isomorphic to  the semilattice of projections $E$ of $Q$. We begin by remarking that from Theorem~\ref{prop:total}  we have that $Q$ is $*$-left Ehresmann with $*$ obtained as in that theorem.

We must first construct the semilattice $\mathcal{X}$.  To do so we show that $T$ acts partially and strongly on $E$ by order-preserving maps, and then call upon the results of Section~\ref{sec:actions}. 

We define a relation $\cdot$  by the rule that for all $[m] \in T$ and $e \in E$,
\[ \exists [m] \cdot e  \mbox{ if } \exists h \in H\mbox{ such that }\ h\, \sigma\, m\ \mbox{and}\ h^\ast \geq e\]
and then
\[[m] \cdot e=(he)^+.\]
We show this is the required partial  action of $T$  on $E$. 

\begin{lem}\label{me+=ue+}
Let  $h, k \in H$ and $e \in E$  such  that $h\, \sigma\, k$ and $h^\ast, k^\ast \geq e$.  Then $he = ke$. Consequently, $\cdot$ is well defined. 
\end{lem}
\begin{proof}  We have  $he \, \sigma\, h\, \sigma\, k \ \sigma\ ke$, and so $he\, \sigma\, ke$. By (H2) we get $he, ke \in H$ and  $(he)^{\ast}  = h^\ast e = e$ as $h^\ast \geq e$. Similarly, $(ke)^{\ast} = e$. Since $H$ is proper $he = ke$ follows.
\end{proof}

\begin{lem}\label{stronglypartialaction2} The pair $(E,\cdot)$  is a strong partial action of  $T$  on the set $E$.
\end{lem}
\begin{proof} The identity of $T$ is $[1]$ and as $E\subseteq H$ and $1=1^*\geq e$ for any $e\in E$, we see that  $[1]\cdot e$ exists and $[1]\cdot e=(1e)^+=e^+=e$. 

If $\exists [t] \cdot e$ and $\exists [s] \cdot ([t] \cdot e)$ for $[s], [t] \in T$ and $e \in E$, then there exist $h, u \in H$ such that $h \, \sigma\, t$, $u\, \sigma\, s$, $h^\ast \geq e$ and $u^\ast \geq [t] \cdot e = (he)^+$. By (H2) we have $he \in H$ and by (H3) we get $uhe \in H$ and  $(uhe)^\ast = (he)^\ast = h^\ast e = e$. Clearly $uhe\, \sigma\, st$. Hence $\exists [st] \cdot e$  and $[st]\cdot e=(uhee)^+=(uhe)^+$.  Further, we have  $[s] \cdot ([t] \cdot e) = [s] \cdot (he)^+ = (u(he)^+)^+ = (uhe)^+ = [st] \cdot e$.

Conversely, if $\exists [t] \cdot e$ and $\exists [st] \cdot e$ for $[s], [t] \in T$ and $e \in E$, then there exist $h, w \in H$ such that
 \[h \, \sigma\, t,\,  h^\ast \geq e\mbox{ and } [t] \cdot e = (he)^+,\]
 and 
 \[w\, \sigma\, st,\,  w^\ast \geq  e\mbox{ and } [st]\cdot e = (we)^+.\]
 By (H2) we know that $he,we\in H$ and $(he)^{\ast} = h^{\ast}e = e = w^{\ast}e = (we)^{\ast}$.  By (H4) there exists $u \in H$ such that $u\, \sigma\, s$.  Then $u(he)\, \sigma\, st\, \sigma\, we $. Moreover, by (H5),   we obtain that there exists $v \in H$ such that $v \, \sigma\,\ u$ and $v^{\ast} \geq (he)^+$, that is, $v \, \sigma\, s$ and $v^{\ast} \geq [t]\cdot e$, which implies that $\exists [s] \cdot ([t] \cdot e)$ as required.
\end{proof}

In Lemma~\ref{stronglypartialaction2}
we explicitly see the role of (H5) in ensuring the partial action is strong.

\begin{lem}\label{stronglypartialaction}  The  pair $(E,\cdot)$  is a strong order-preserving partial action of  $T$  on the semilattice $E$.   For any $[t]\in T$ there is an $e\in E$ such that $\exists [t]\cdot e$.
\end{lem}
\begin{proof} 
Suppose that $[t]\in T$ and $e, f\in E$ with $e \leq f$ and  $\exists [t] \cdot f$. Then there exists $h \in H$ such that $t\, \sigma\, h$ and $h^\ast \geq f$.   Thus   $h^{\ast} \geq e$ and so $\exists [t] \cdot e$.

Further,  we have $[t] \cdot f \geq [t] \cdot e$ since $[t] \cdot f = (hf)^+ \geq (he)^+ =[t] \cdot e$, by Lemma~\ref{lem:naturalaction}.

To see that the final statement holds, let $[t]\in T$. There exists $h\in H$ with $t\,\sigma\, h$ and clearly
$[t]\cdot h^*$ exists.
\end{proof}

From Theorem~\ref{them:slglobal}  the partial action $(E, \cdot)$ admits a globalisation $(\kappa, \mathcal{X}, \bullet)$ where $\mathcal{X}$ is a semilattice with identity, $\kappa:E\rightarrow \X$ is a semilattice embedding,   and $(\mathcal{X},\bullet)$ is an order-preserving action  of $T$. We maintain the notation of Theorem~\ref{them:slglobal}, so that $e\kappa=(1,e)^\omega$.
For convenience, we remind the reader that since  $(\kappa, \mathcal{X}, \bullet)$ is a globalisation of $(E,\cdot)$, and utilising the definition of $\kappa$, the following statements are equivalent for all $[t]\in T$ and $e\in E$:

\[
\exists [t]\cdot e;\,\,\,  [t]\bullet e\kappa\in E\kappa;\,\,\, 
[t]\bullet (1,e)^\omega=(1,f)^\omega\mbox{ , for some }f\in E\]
and then if any of these statements are true
\[ ([t]\cdot e)\kappa=(1, [t]\cdot e)^\omega=[t]\bullet (1,e)^\omega,\]
so that $f=[t]\cdot e.$

Let $\Y=E\kappa$.
 Lemma~\ref{lem:andb}  follows from Lemma~\ref{stronglypartialaction} and  the fact that $(\kappa, \mathcal{X}, \bullet)$ is a globalisation of $(E, \cdot)$.

\begin{lem}\label{lem:andb}
The action of $T$ on $\mathcal{X}$ with respect to the subsemilattice $\Y$ satisfies conditions (A) and (B).
\end{lem}
We may now define $\mathcal{P}:=\mathcal{P}_{\ell}(T,\X)$ 
and the $*$-left Ehresmann monoid  
$$\QQ:=\QQ(T, \X,\Y)=\langle H^{\mathcal{Q}}\rangle_{(2)}$$
where 
\[H^{\mathcal{Q}}=\{ [t](1,e)^\omega: (1,e)^\omega,   [t]\bullet (1,e)^\omega \in \mathcal{Y}\}=\{[t](1,e)^\omega: \exists [t]\cdot e\}. \]

\begin{lem}\label{lem:htrick} We have that
\[H^{\mathcal{Q}}=\{ [h](1,h^*)^\omega:h\in H\},\]
and for $h\in H,$ 
\[([h](1,h^*)^\omega)^+=h^+\kappa\mbox{ and }
([h](1,h^*)^\omega)^*=h^*\kappa.\]
\end{lem}
\begin{proof} Let $h\in H$; certainly $[h]\cdot h^*$ is defined so that $[h](1,h^*)^\omega\in H^{\mathcal{Q}}$.
Conversely, if $[t](1,e)^\omega\in H^{\mathcal{Q}}$ then
$\exists [t]\cdot e$ and so $t\,\sigma\, k$ for some $k\in H$ with $k^*\geq e$. Then for 
$h=ke$ we have  $t\,\sigma\, h$  and  $h^*=k^*e=e$. Hence $[t](1,e)^\omega=
[h](1,h^*)^\omega$ and the first statement holds. 

Now let  $h\in H$. Observe that $[h]\cdot h^*=(hh^*)^+=h^+$. Then
\[([h](1,h^*)^\omega)^+=[h]\bullet (1,h^*)^\omega=
(1,[h]\cdot h^*)^\omega=(1,h^+)^\omega=h^+\kappa.\]
Finally, 
\[([h](1,h^*)^\omega)^*=(1,h^*)^\omega=h^*\kappa,\]  
as required.
\end{proof}

\begin{lem}\label{lem:theta} Define $\theta:H\rightarrow H^{\mathcal{Q}}$ by  $h\theta=[h](1,h^*)^\omega$.
Then $\theta$ is a bijection such that 
$h^*\theta=(h\theta)^*$ and $h^+\theta=(h\theta)^+$.
\end{lem}
\begin{proof} To see that $\theta$ is one-one, suppose that
$h\theta=k\theta$ for some $h,k\in H$. Then
$[h](1,h^*)^\omega=[k](1,k^*)^\omega$. From Lemma~\ref{lem:HH} we deduce that $[h]=[k]$ so that $h\,\sigma\, k$, and $(1,h^*)^\omega=(1,k^*)^\omega$. The latter statement says that $h^*\kappa=k^*\kappa$ and as $\kappa$ is an embedding, we  obtain $h^*=k^*$. The fact that $h=k$ now comes from $H$ being proper. 

For any $e\in E$ notice that 
\[e\theta=[e](1,e)^\omega=(1,e)^\omega=e\kappa.\]
The lemma now follows from Lemma~\ref{lem:htrick}.\end{proof}

\begin{cor}\label{cor:nf} An element $h_1\dots h_n$ of $Q$ is in $H$-canonical form if and only if 
$(h_1\theta)\dots (h_n\theta)$ is in $H^{\mathcal{Q}}$-canonical form.
\end{cor}
\begin{proof} First, observe that for $h,k\in H$, we have 
\[h^*<k^+\Leftrightarrow h^*\kappa<k^+\kappa\Leftrightarrow h^*\theta<k^+\theta\Leftrightarrow (h\theta)^*<(k\theta)^+,\]
 using Lemma~\ref{lem:theta}. 

 If $h\in E$, then $h\theta=h\kappa\in \Y$. On the other hand, if $h\theta=[h](1,h^*)^\omega\in \Y$, then by Lemma~\ref{lem:HH} we get that $[h]$ is the  identity of $Q/\sigma$, that is, $[h]=[1]$. Hence  
 Corollary~\ref{cor:hSigma1_M} gives that $h\in E$. The result follows.
    \end{proof}

 We are now in a position to state the main result of this section.

\begin{them}\label{them:main}   Let $Q$
be a left Ehresmann monoid with proper  basis $H$. Then $Q$ is isomorphic to  a  biunary monoid subsemigroup  $\QQ(T,\X,\Y)$ of some $\mathcal{P}_{\ell}(T,\X)$. 

\end{them}

\begin{proof} Let  $\QQ=\mathcal{Q}(T, \X , \Y)$ and $\theta:H\rightarrow H^{\mathcal{Q}}$ be the bijection defined as above. We extend the domain of $\theta$ to $Q$ by  defining 
\[(h_1\dots h_n)\theta=(h_1\theta)\dots (h_n\theta)\]
where $h_1\dots h_n$ is in $H$-canonical form. Lemma~\ref{lem:theta}  and Corollary~\ref{cor:nf} together with the fact that $H$ is a basis for $Q$ and $H^{\mathcal{Q}}$ is a basis for $\QQ$ give  us that this extended $\theta$ is also a bijection.

To show that $\theta$ is a  semigroup morphism, we proceed by induction to show  that $(h_1\dots h_n)\theta=(h_1\theta)\dots (h_n\theta)$ for any $h_i \in H$, $i = 1, \dots, n$, without the restriction that $h_1\dots h_n$ is in $H$-canonical form.

  Clearly, the statement  holds for $n =1$.  We now prove the case $n=2$, that is, we verify $(h_1h_2)\theta = (h_1\theta)( h_2\theta)$ for any $h_1,h_2\in H$.  By Lemma~\ref{h+m-Hormal form}, there are three  different possibilities when looking for the $H$-canonical form for $h_1h_2$

  (1) Suppose that $h_2\in E$. Then $[h_2]$ is the identity of $\QQ$,  and  by (H2)
we have $(h_1h_2)^* = h_1^*h_2= h_1^* h_2^*$ . Notice that  $$(1, h_1^*)^\omega(1, h_2^*)^\omega =(1, h_1^*h_2^*)^\omega$$ 
since $\kappa$ is a semilattice embedding. 
Hence
\[
\begin{aligned}
(h_1\theta)( h_2\theta) &= [h_1](1, h_1^*)^\omega[h_2](1, h_2^*)^\omega=[h_1](1, h_1^*)^\omega(1, h_2^*)^\omega=[h_1](1, h_1^*h_2^*)^\omega\\
        &=[h_1][h_2](1, h_1^*h_2^*)^\omega=[h_1h_2](1,(h_1h_2)^*)^\omega=(h_1h_2)\theta.\
\end{aligned}
\]

(2)
If $h_2\notin E$  and $h_1^* h_2^+= h_2^+$  then $h_1h_2 \in H$ and $(h_1h_2)^* = h_2^*$ by (H3). In the following have in mind that by Lemma~\ref{lem:htrick}  we  get $([h_2](1, h_2^*)^\omega)^+ = (1,h_2^+)^\omega$, and $\kappa$ is a semilattice embedding. We have   
\[
\begin{aligned}
(h_1h_2)\theta & = [h_1h_2](1, (h_1h_2)^*)^\omega
               =[h_1][h_2](1, h_2^*)^\omega=[h_1]([h_2](1, h_2^*)^\omega)^+[h_2](1, h_2^*)^\omega\\
               &=[h_1](1,h_2^+)^\omega[h_2](1, h_2^*)^\omega=[h_1](1,h_1^*)^\omega(1,h_2^+)^\omega[h_2](1, h_2^*)^\omega\\
               &=[h_1](1,h_1^*)^\omega[h_2](1, h_2^*)^\omega=(h_1\theta )(h_2\theta).
\end{aligned}
\]

(3)  If $h_2 \notin E$ and  $h_1^* h_2^+< h_2^+$,  then  Lemma~\ref{h+m-Hormal form} tells that  $(h_1h_2^+)h_2$ is in $H$-canonical form.  Using  the definition of $\theta$ and the case (1),   we have  \[(h_1h_2)\theta = ((h_1h_2^+)h_2)\theta =((h_1h_2^+)\theta )(h_2\theta)=(h_1\theta)(h_2^+\theta)(h_2\theta)\]
and then 
by Lemma~\ref{lem:theta},
\[(h_1h_2)\theta=(h_1\theta)(h_2\theta)^+(h_2\theta)=(h_1\theta)(h_2\theta).\]

Suppose for induction that  the statement holds  for $n-1$, where $n\geq 2$,  and consider a product
$h_1\dots h_n$, where $h_i\in H$ for $1\leq i\leq n$. 
Let $h_2\dots h_n = k_2\dots k_m$ where $k_2\dots k_m$  is in $H$-canonical form. Then $(h_1h_2\dots h_n)\theta = (h_1k_2\dots k_m)\theta$ and by the induction hypothesis,  $(h_2\dots h_n)\theta=(h_2\theta)\dots (h_n\theta)= (k_2\theta) \dots (k_m\theta)$. Now we look the $H$-canonical form of $h_1h_2\dots h_n=h_1k_2\dots k_m$. 
 
By Lemma~\ref{h+m-Hormal form}, this can be:

(1) $(h_1k_2)k_3\dots k_m$, in which case  
$(h_1h_2\dots h_n)\theta=  (h_1k_2)\theta (k_3\theta) \dots (k_m\theta)=(h_1\theta)( k_2\theta) (k_3\theta) \dots (k_m\theta)=(h_1\theta)(h_2\theta) \dots (h_n\theta)$ using the  $n=2$ case  and we are done too; or finally

(2) $(h_1k_2^+)k_2\dots k_m$, in which case 
$(h_1h_2\dots h_n)\theta=(h_1k_2^+)\theta (k_2\theta )\dots( k_m\theta)=(h_1\theta)(k_2^+\theta) (k_2\theta )\dots( k_m\theta)$ using the  $n=2$ case again. But $k_2^+\theta =(k_2\theta)^+$ so that $(h_1h_2\dots h_n)\theta= (h_1\theta) (k_2\theta )\dots( k_m\theta)=(h_1\theta)(h_2\theta) \dots (h_n\theta)$ as required.

It is now clear that $\theta$ is a semigroup isomorphism and as such  must preserve the identity. To show that $\theta$ respects the unary operation $+$, let $h_1h_2\dots h_n\in Q$ be in $H$-canonical form. By Lemma~\ref{prop} $(i)$, Lemma~\ref{lem:theta} and Corollary~\ref{cor:nf}, we have
\[
(h_1h_2\dots h_n)^+\theta= h_1^+\theta=(h_1\theta)^+=((h_1\theta)( h_2\theta) \dots (h_n\theta))^+= ((h_1h_2\dots h_n)\theta)^+.
\]
The argument that  $\theta$ respects the unary operation $*$ is similar.
\end{proof}

\section{Main theorem and further directions}\label{sec:main}

 We now  draw together the results of previous sections and proceed to point out a number of directions for further research. Note that (3) of Theorem~\ref{them:conclusion}  indicates how to find subsemigroups $\QQ(T,\X,\Y)$ of a $\mathcal{P}_{\ell}(T,\X)$ in terms of the  generating set $H^{\mathcal{P}}$ of  $\mathcal{P}_{\ell}(T,\X)$, rather than from the data given to construct $\QQ(T,\X,\Y)$ in Section~\ref{sec:Q}.

\begin{them}\label{them:conclusion} The following conditions are equivalent for a biunary monoid $Q$:
\begin{enumerate}
\item  $Q$
is  left Ehresmann  and $Q=\langle H\rangle_{(2)}$ for a proper basis $H$;
\item  $Q$ is isomorphic to a biunary monoid subsemigroup $\QQ(T,\X,\Y)$ of some $\mathcal{P}_{\ell}(T,\X)$;
\item $Q$ is isomorphic to a $*$-left Ehresmann monoid subsemigroup $\QQ=\langle H^{\mathcal{Q}}\rangle_{(2)}$ of some $\mathcal{P}_{\ell}(T,\X)$
where $H^{\mathcal{Q}}=\{ h\in H^{\mathcal{P}}: h^+,h^*\in \Y\}=H^{\mathcal{P}}\cap \QQ$ and $H^{\mathcal{Q}}/\sigma^{\mathcal{P}}=\mathcal{P}_{\ell}(T,\X)/\sigma^{\mathcal{P}}$.
\end{enumerate}
\end{them}
\begin{proof} The equivalence of (1) and (2) is given by  Theorems~\ref{them:Q} and \ref{them:main}.

(2) implies (3). As observed,  $\QQ$  is a $*$-left Ehresmann monoid.   Further,   $\QQ=\langle H^{\mathcal{Q}}\rangle_{(2)}$ and, using Lemma~\ref{lem:HH}, we have
 \[H^{\mathcal{Q}}= \{ te\in H^{\mathcal{P}}: e\in\Y, \, t\cdot e\in \Y\}=\{ h\in H^{\mathcal{P}}: h^*, h^+ \in\Y\}\] 
where $H^{\mathcal{Q}}= H^{\mathcal{P}}\cap \QQ$ from Lemma~\ref{lem:lucky}.
 From Theorem~\ref{P1} we know that $\mathcal{P}_{\ell}(T,\X)/\sigma^{\mathcal{P}}=\overline{T}$,  where $\overline{T}$ consists of the $\sigma^{\mathcal{P}}$-classes of elements of $T$. For any $t\in T$ by condition (B) of Theorem~\ref{them:Q}, there exists a $te\in H^{\mathcal{Q}}$ where $e\in \Y$, and then
$te\,\sigma^{\mathcal{P}}\, t$.

(3) implies (2). Let $\QQ$ be as given. Then the semilattice of projections $\mathcal{Y}$ of $\QQ$ is a monoid subsemilattice of $\X$.
Let $t\in T$. Let $e,f\in \Y$ with $f\leq e$ and suppose that 
$ t\cdot e\in \Y$.  Then $te\in H^{\mathcal{Q}}\subseteq \QQ$, whence  $tef=tf\in \QQ$. Thus $tf\in H^{\mathcal{P}}\cap \QQ=H^{\mathcal{Q}}$, and so $(tf)^+=t\cdot f\in \Y$. Therefore (A) holds. By  the final assumption, $t\, \sigma^{\mathcal{P}}\, h$ for some $h\in H^{\mathcal{Q}}$.
But we must then have $h=te$ for some $e\in\Y$ where $t\cdot e\in \Y$. Thus (B) holds. Hence $\QQ(T,\X,\Y)$ is exactly our $\QQ$ as given.
\end{proof}

We have focused in this article entirely on left Ehresmann {\em monoids}. In many places, the existence of the identity is important in our arguments, not least in the construction of monoids $\mathcal{P}_\ell(T,X)$. However, there is always the device of adjoining an identity that can be employed.

\begin{question}\label{qn:semigroups} Develop a corresponding theory of left Ehresmann semigroups.
\end{question}

A left Ehresmann monoid $M$ acts naturally by order-preserving maps on its semilattice of projections. Insisting that this action is by morphisms corresponds to the identity $(xy^+z^+)^+=(xy)^+(xz)^+$ holding. Left Ehresmann monoids satisfying this identity are called {\em left weakly $E$-hedged} \cite{FGG2}.  If a left Ehresmann monoid $M$ satisfies the quasi-identities
$x^2=x\rightarrow x=x^+$ and $yx=zx\rightarrow yx^+=zx^+$ then it is called {\em left adequate}; a left adequate monoid satisfying
$(xy^+z^+)^+=(xy)^+(xz)^+$ is  called {\em left h-adequate} \cite{F4}.  If the ample identity $xy^+=(xy)^+x$ holds in a left Ehresmann monoid,   then  so does the identity $(xy^+z^+)^+=(xy)^+(xz)^+$. We remark that 
free left h-adequate monoids are determined in \cite{F4}; they do not satisfy the ample identity.

\begin{question}\label{qn:hsemigroups} Develop a corresponding theory for left Ehresmann monoids and semigroups where the natural action on the semilattice of projections is by morphisms.
\end{question}

A theory of (two-sided) Ehresmann semigroups is developed by Kudravtseva and Laan in \cite{KL}, where the authors declare they start from a different viewpoint from ours. 
 They develop a theory of {\em proper} Ehresmann semigroups using a notion of {\em matching factorisations} of what they call {\em proper} elements. As previously indicated, our notion of proper elements is close to theirs.  The  approach of \cite{KL} is based on categories. For an Ehresmann semigroup, the kernel of the unary operation $*$ is a right congruence, whereas in our case it is not. Nevertheless, we ask the following.

\begin{question}\label{qn:2sided} What are the deeper connections of our approach with that of \cite{KL}? 
\end{question}

\end{document}